\documentclass[a4paper,11pt, twoside]{article}
\usepackage{amsmath,amsthm}
\usepackage{amssymb,latexsym}
\usepackage{mathrsfs}
\usepackage{enumerate}
\usepackage[colorlinks,citecolor=blue]{hyperref}
\usepackage[numbers,sort&compress]{natbib}
\usepackage{indentfirst}
\usepackage{color}
\usepackage{esint}

\DeclareMathAlphabet{\mathpzc}{OT1}{pzc}{m}{it}

\newtheorem{theorem}{Theorem}[section]

\newtheorem{lemma}{Lemma}[section]

\newtheorem{assumption}{Assumption}[section]
\theoremstyle{definition}
\theoremstyle{remark}

\numberwithin{equation}{section}
\allowdisplaybreaks

\date{}

\begin{document}

\markboth{ }{Weighted estimates for time fractional parabolic equations}

\date{}
\baselineskip 0.22in

\title{Weighted estimates for time-fractional parabolic equations with VMO coefficients}
\author{Jia Wei He$^{1,2 ,}$\footnote{E-mail: jwhe@gxu.edu.cn } , Lu Lu Tao$^{1}$ \\[1.8mm]
\footnotesize  {1.School of Mathematics and Information Science, Guangxi University, Nanning 530004, P.R. China}\\[1.5mm]
\footnotesize  {2. Guangxi Center for Mathematical Research, Guangxi University, Nanning 530004, P.R. China}\\[1.5mm]}
\date{}
\maketitle

\begin{abstract}
This paper is devoted to the weighted estimates and the solvability of time-fractional parabolic equations. The leading coefficients \(a^{ij}(t,x)\) are assumed to have small mean oscillations in \((t,x)\) locally, in both non-divergence and divergence forms, in the whole space. By employing appropriate odd and even extensions along with suitable boundary value conditions, we derive the corresponding results for the half-space. The proofs rely on the application of the Fefferman-Stein theorem and the Hardy-Littlewood maximal function theorem in the context of weighted mixed spaces.
\end{abstract}

\baselineskip 0.25in
\section{Introduction}
We focus on parabolic equations with  time fractional derivative in nondivergence form
\begin{equation}\label{e1.1}
-\partial_t^\alpha u+a^{ij}(t,x)D_{ij}u+b^i(t,x)D_iu+c(t,x)u=f ,
\end{equation}
and in divergence form
\begin{equation}\label{e1.2}
-\partial_t^\alpha u+D_{i}\left(a^{i j} D_{j} u+a^{i} u\right)+b^{i} D_{i} u+c u=D_if +g ,
\end{equation}
 with given zero initial conditions $u(0,\cdot)=0$ within a cylinder domain $(0,T)\times\Omega$, as the whole space $\mathbb{R}^{d}$ and the half space $\mathbb{R}_+^d:=\{(x^1,\ldots,x^d)\in\mathbb{R}^{d}:x^1>0\}$, where $\partial_t^\alpha $ represents the Caputo fractional derivative of order $ \alpha\in(0,1)$. 

Equations \eqref{e1.1} and \eqref{e1.2} are classified as a time nonlocal type parabolic equation when \(0 < \alpha < 1\) and as a fractional wave equation when \(1 < \alpha < 2\). The exploration of these equations is driven by their importance in the field of physics, fractional parabolic equations are essential for modeling anomalous diffusion on fractals and have broad applications across physics and probability theory, these equations are particularly significant in areas such as continuous-time random walks \cite{me1}, certain types of amorphous semiconductors, highly porous materials \cite{an, me2} and other related researches. The fractional wave equation will govern intermediate processes between diffusion and wave propagation in viscoelastic media \cite{Mainardi10}. Equations \eqref{e1.1} and \eqref{e1.2} are associated with the classical Volterra-type equation \cite{pr}, which is an integral equation involving the convolution with the kernel being a ramp function or a power function
$$u(t,x)+\int\limits_{0}^{t}a(t-s)Av(s,x) ds=f(t,x),$$
where $A$ is a sectorial operator independent of time.
Observe that, Zacher \cite{za} obtained a unique solution for this type of Volterra equation by applying operator theory.
For the version of evolution equations, Luchko and Yamamoto \cite{lu} discussed initial-boundary value problems and obtained the uniqueness of both strong and weak solutions by the maximum principle. Kian, et al. \cite{kian} considered an inverse problem for time fractional diffusion equations and obtained its global uniqueness of the Riemannian metric. Thereafter,  Helin et al. \cite{Helin} extended the results to space-time fractional diffusion equations connected compact Riemannian manifold without boundary.
Oka et al. \cite{oka} deduced existence of solutions for time fractional semilinear parabolic
equations in Besov–Morrey spaces. 
The fractional equations \eqref{e1.1} and \eqref{e1.2}, which depend on time and spatial variables, can be abstracted to form nonautonomous fractional evolution equations. This abstraction is considered in the case where the coefficients are  H\"{o}lder continuous in the time variable and have an $L^p$-measure in the spatial variable, with the domain being of the invariant constant type, the problem admits a strong solution, as proven by El-Borai \cite{El-Borai}, and a classical solution, as shown by Zhou and He \cite{Zhou-He24} as well as the almost sectorial operator in \cite{He-Zhou24}. Bajlekova \cite{Bajlekova} obtained the $L_p$-maximal regularity of solutions. This was also demonstrated by Mahdi \cite{Mahdi} in the setting of Hilbert spaces. When $A(t) = A + B(t)$, where $A$ is the infinitesimal generator of an $\alpha$-resolvent family and $B$ is a strongly continuous operator satisfying the Lipschitz condition, Henr\'{i}quez et al. \cite{Hernan} established the existence of solutions for order $\alpha \in (1, 2)$. Recently,  Cort\'{a}zar et al. \cite{co} considered the large-time behavior of solutions for nonhomogeneous heat equations with memory.

Note that, in the limit case 
\(\alpha = 1\), Krylov \cite{kr1} developed an 
\(L_p\)-theory for elliptic and parabolic equations in divergence and non-divergence forms with VMO coefficients. The main conclusions were further explored with mixed norms in \cite{kr2}. Dong and Kim \cite{dong} derived mixed-norm weighted 
$L_p$-estimates for elliptic and parabolic equations with (partially) BMO coefficients by utilizing Fefferman-Stein sharp functions theorems and $A_p$
weights. The related results on the coefficient problem of elliptic and parabolic equations, see, e.g. \cite{by1, dong, dong1, dong2, hal, ki, kim, kr1, kr3} and references therein. 
These coefficient discussions have been extended to fractional problems. Kim et al. \cite{ki} showed the $L_p$-$L_q$ estimates of solutions and worked on the existence results with variable coefficients $a^{ij}(t,x)$ piecewise continuous in $t$ and uniformly continuous in $x$ for the form:
\[
\partial_t^\alpha u(t,x) = a^{ij}(t,x)u_{x^ix^j}(t,x) + f(t,x,u), \quad t > 0, \quad x \in \mathbb{R}^d,
\]
where \(\alpha \in (0, 2)\). Dong and Kim \cite{dong4} noted that $L_p$-solvability is indeed just required with coefficients merely measurable in $t$ and small mean oscillations in $x$, which is also known as VMO coefficients. Meanwhile, the authors in \cite{dong5} complemented the results of \cite{dong4} and studied divergence form parabolic equations for \(\alpha \in (0, 1)\) with VMO coefficients and 
$$-\partial_t^\alpha u+D_i(a^{ij}D_ju+a^iu)+b^iD_iu+cu=D_ig_i+f$$
in the whole space and in the half space. They focused on the case where the main coefficients are measurable in $t$ and small mean oscillations in $x$, proving the solvability in \(L^q(L^p)\) spaces when \(q \geq p\). This work generalizes previous results to cover time-dependent coefficients and contributes to the understanding of parabolic equations with VMO coefficients.
 In the special case that coefficients $a^{ij}(t,x)=\delta_{ij}$, Han el at. \cite{han} introduced the Muckenhoupt weights for the fractional diffusion-wave equation
$$\partial_t^\alpha u(t,x)=\Delta  u(t,x)+f(t,x),\quad t>0,~~x\in\mathbb{R}^{d},$$
where $\alpha\in(0,2)$, they obtained the solvability of solutions with weighted $L_p$-$L_q$ theory. 
Recently, Dong and Kim considered nondivergence form fractional parabolic equations with VMO and obtained significant weighted mixed-norm estimates, as detailed in \cite{dong1}. Subsequently, Dong and Liu \cite{dong2} presented weighted \(L_p\) estimates for fractional wave equations across various domains, including the whole space, a half space, and cylindrical domains. These results emphasize the importance of weighted mixed-norm estimates and the solvability of equations in both nondivergence and divergence forms. And the difference between the leading coefficients $a^{ij}$ of \cite{dong1} and \cite{dong2}  is that the coefficients of the former are measurable in the time variable and small mean oscillations in the spatial variables, while those of the latter have small mean oscillations of $(t,x)$. Then, Dong and Kim \cite{dong6} studied time fractional parabolic equations in divergence and non-divergence forms
with partially SMO coefficients. The leading coefficients $a^{ij}$ are merely measurable in $(t,x_1), 1\leq i,j\leq d, (i,j)\neq(1,1)$ and the coefficient $a^{11}$ is merely measurable locally either in $t$ or 
$x_1$. They obtained the unique solvability in Sobolev spaces by using $L_{p_0}$-mean oscillation estimates with $p_0 \in (1,\infty)$.

Building on the previous works, 
it is notice that the results of fractional parabolic equations can be further improved, the contributions of this paper could be summarized as follows: Let $p,q\in(1,\infty)$, define the weight $w:=w_1(t)w_2(x)$ where $w_1(t)\in A_p(\Omega)$ and $w_2(x)\in A_q(\Omega)$. For $f\in L_{p,q,w}\left((0, T)\times\Omega\right)$, we show that there is a unique solution $u$ in the weighted space, which satisfies  fractional parabolic equation \eqref{e1.1} of non-divergence form  in the half space with the zero initial condition such that
$$\|\partial_t^\alpha u\|_{L_{p,q,w}}+\sum^{2}_{j=0}\|D^ju\|_{L_{p,q,w}}\lesssim \|f\|_{L_{p,q,w}}.$$
Both in the whole space and a half space, we establish the existence of a unique solution $u$ in the weighted space for equation \eqref{e1.2} in divergence form as
$$\left\|\partial_{t}^{\alpha} u\right\|_{\mathbb{Z}_{p, q, w}^{-1}}+\sum^{1}_{j=0}\|D^ju\|_{L_{p,q,w}}\lesssim\|f_i\|_{L_{p,q,w}}+\|g\|_{L_{p,q,w}}.$$
In this paper, we consider the weighted estimates and the solvability of time-fractional
parabolic equations with WMO coefficients that can be seen as a special case of \cite{dong6}. It is worth noting that we use  different methods, achieve results consistent with \cite{dong6} and give proofs of the correlation in the half space.

The paper is organized as follows. In Section \ref{sec2}, we provide some definitions of key symbols, give function spaces, and present our main results. Then, the Weight mixed-norm estimates in the half space for equation \eqref{e1.1} and the proof of Theorem \ref{the2.1} are detailed in Section \ref{sec3}. Section \ref{sec4} covers weight mixed-norm estimates for equation \eqref{e1.2} in the whole space, along with the proof of Theorem \ref{the2.2}. Finally, Section \ref{sec5} proves a similar result for the half space.

\section{Preliminaries}\label{sec2}
We first present some notation and facts related to the fractional calculus. We denote the Riemann-Liouville fractional integral of order $\alpha\in(0,1)$ is defined by
$$I_a^\alpha v=\frac{1}{\Gamma(\alpha)}\int_a^t(t-s)^{\alpha-1}v(s)\:ds,$$
for $-\infty<a<t$ and the Caputo fractional derivative is given by
$$D_t^\alpha v=\frac{1}{\Gamma(1-\alpha)}\partial_t\int_0^t(t-s)^{-\alpha}v(s )\:ds=\partial_tI_0^{1-\alpha}v,$$
we have $\partial_t^\alpha v(t)=D_t^\alpha\left(v(t )-v(0 )\right)$ and we have $\partial_t^{\alpha}v=D_t^\alpha v$
for smooth functions $v$ with $v(0 )=0$.  We use $\partial_t^\alpha v$ to indicate $\partial_tI_a^{1-\alpha} v$ for $v(\cdot )=0$.

The parabolic cylinder is given by
\begin{equation*}
Q_{R_{1},R_{2}}(t,x)=(t-R_{1}^{\frac{2}{\alpha}},t)\times B_{R_{2}}(x),
\end{equation*}
where $B_{R_{2}}(x)=\{z\in\mathbb{R}^{d}:|z-x|<R_{2}\}$. Particularly, $Q_R(t,x):=Q_{R,R}(t,x)$, $B_R:=B_R(0)$ and $Q_R:=Q_R(0,0)$. 

For $p\in(1,\infty)$, $w$ is non-negative function and the set $A_p(\mathbb{R}^{d},dy)$ satisfy
$$[w]_{A_{p}(\mathbb{R}^{d})}:=\sup\limits_{y^{\ast}\in\mathbb{R}^{d},R>0}
\left(\displaystyle\fint_{B_R(y^{\ast})}\:w(y)\:dy\right)
\left(\displaystyle\fint_{B_R(y^{\ast})}\:(w(y))^{\frac{-1}{p-1}}\:dy\right)
^{p-1}<+\infty,$$
where $B_R(y^{\ast})=\{y\in\mathbb{R}^d:|y-y^{\ast}|<R\}$ and clearly $[w]_{A_{p}}\geq1$.

Let $T\in(0,\infty)$, $\mathfrak{R}^d:=(0, T)\times \mathbb{R}^{d}$, and $\mathfrak{R}_{+}^d:=(0, T)\times \mathbb{R}^{d}_{+}$. For a constant $\varrho>0$, if $w=w_1(t)w_2(x)$ with $w_1(t)\in A_p(\mathbb{R},dt)$, $w_2(x)\in A_q(\mathbb{R}^{d},dx)$, $[w_1]_{A_p}\leq \varrho$, and $[w_2]_{A_q(\mathbb{R}^{d})}\leq \varrho$, we write $[w]_{p,q}\leq \varrho$ in the whole space and $[w]_{p,q,\mathbb{R}^{d}_{+}}\leq \varrho$ in the half space. Then, the set  $L_{p,q,w}\left(\mathfrak{R}^{d}\right)$ is defined with
$$\|f\|_{L_{p,q,w}(\mathfrak{R}^{d})}:=\left(\int_0^T\left(\int_{\mathbb{R}^{d}}|f(t,x)|^q\:w_2(x)dx\right)^{p/q}w_1(t)\:dt\right)^{1/p}<+\infty,$$
for measurable functions $f$, where $w_1\in A_p(\mathbb{R},dt)$ and $w_2\in A_q(\mathbb{R}^{d},dx)$ on $\mathfrak{R}^{d}$.

We denote $v \in \mathbb{Z}_{p, q, w, 0}^{\alpha, 2}\left(\mathfrak{R}^d\right)$ if there is a sequence $ v_{n} \in C_{0}^{\infty}\left([0, T] \times \mathbb{R}^{d}\right)$ vanishing for large $|x|$, $v_{n}(0, x)=0$ and as $n\to\infty$, 
\begin{equation}\label{e2.1}
\|v_n-v\|_{\mathbb{Z}_{p,q,w}^{\alpha,2}(\mathfrak{R}^d)}:=\|\partial_t^\alpha v_n-\partial_t^\alpha v\|_{L_{p,q,w}(\mathfrak{R}^{d})}+\sum_{j=0}^2\|D^jv_n-D^jv\|_{L_{p,q,w}(\mathfrak{R}^{d})}\rightarrow0.
\end{equation}

Particularly, $L_{p, q, w}\left(\mathfrak{R}^{d}\right):= L_{p}\left(\mathfrak{R}^{d}\right)$ when $p=q$ and $w=1$ and we write $ \mathbb{Z}_{p, w, 0}^{\alpha, 2}\left(\mathfrak{R}^{d}\right)$ for $p=q$. For any $R>0$, we denote $u \in \mathbb{Z}_{p, 0, \text{loc}}^{\alpha, 2}\left(\mathfrak{R}^{d}\right):=\mathbb{Z}_{p, 0}^{\alpha, 2}\left((0, T) \times B_{R}\right)$, $f\in L_{p, \text{loc}}\left(\mathfrak{R}^{d}\right):= L_p\left((0,T)\times B_R\right)$. Moreover, in the half space, writing $v \in \widetilde{\mathbb{Z}}_{p, q, w, 0}^{\alpha , 2}\left(\mathfrak{R}^d_{+}\right)$  if  $v \in \mathbb{Z}_{p, q, w, 0}^{\alpha, 2}\left(\mathfrak{R}^d_{+}\right)$ such that  $v=0$ on $\partial_{p}((0, T) \times  \mathbb{R}^{d}_{+})$. 

In the sequel, we introduce the function spaces with equations in divergence form. Denote $\partial _t^\alpha v\in \mathbb{Z} _{p, q, w}^{-1}( (a, T) \times \mathbb{R}^{d})$  satisfying 
$$-\partial_t^\alpha v=D_jf_j+g,$$
for some  $v, f_j, g\in L_{p,q,w}((a,T)\times\mathbb{R}^{d}), j=1,\ldots,d$. In this way, it yields
\begin{equation}\label{e2.2}
\int_{a}^T\int_{\mathbb{R}^{d}}\partial_t^\alpha v\psi\:dxdt=\int_{a}^T\int_{\mathbb{R}^{d}}(g\psi-f_jD_j\psi)\:dxdt,
\end{equation} 
for any $\psi\in C_0^{\infty}([a,T)\times\mathbb{R}^{d})$, 
with the norm
\begin{equation*}
\|\partial_{t}^{\alpha}v\|_{\mathbb{Z}_{p,q,w}^{-1}( (a, T) \times \mathbb{R}^{d})}=\inf\Big\{\sum_{j=1}^d\ \|f_j\|_{L_{p,q,w}((a,T) \times\mathbb{R}^d)}+ \|g\|_{L_{p,q,w}((a,T) \times\mathbb{R}^d)}: \eqref{e2.2} ~ \mathrm{is~satisfied}\Big\}.
\end{equation*}
Following this, we denote $v \in\mathcal{Z}_{p, q, w}^{\alpha, 1}((a, T) \times \mathbb{R}^d)$ with norm
$$\|v\|_{\mathcal{Z}_{p,q,w}^{\alpha,1}((a,T)\times\mathbb{R}^d)}=\|\partial_t^\alpha v\|_{\mathbb{Z}_{p,q,w}^{-1}((a,T)\times\mathbb{R}^d)}+\sum_{i=0}^1\|D^iv\|_{L_{p,q,w}((a,T)\times\mathbb{R}^d)},$$
for some $v,Dv\in L_{p, q, w}((a, T) \times\mathbb{R}^d)$ and $\partial_t^\alpha v\in\mathbb{Z}_{p,q,w}^{-1}((a,T)\times\mathbb{R}^d)$.
Writing $v\in\mathcal{Z}_{p,q,w,0}^{\alpha,1}((a,T)\times\mathbb{R}^d)$ if there exists a sequence $\{v_n\}\subset C^\infty([a,T]\times\mathbb{R}^d)$ vanishing for large $|x|,v_n(a,x)=0$, and as $n\to\infty$,
$$\|v_n-v\|_{\mathcal{Z}_{p,q,w}^{\alpha,1}((a,T)\times\mathbb{R}^d)}\rightarrow 0.$$


Moreover, for any $(t,x)\in\mathbf{D}\subset\mathbb{R}^{d+1}$, $f\in L_{1,\mathrm{loc}}$, we introduce maximal functions
$$\mathcal{M}f(t,x)=\sup_{Q_R(v,z)\ni(t,x)}\fint_{Q_R(v,z)}|f(R,y)|\chi_{\mathbf{D}}\:dydR$$
and strong maximal functions
$$(\mathcal{SM}f)\left(t,x\right)=\sup_{Q_{R_1,R_{2}}(v,z)\ni(t,x)}\fint_{Q_{R_1,R_{2}}(v,z)}|f(R,y)|\chi_{\mathbf{D}}\:dydR.$$

\begin{assumption}\label{assump2.1}
There exists $\tau\in(0,1)$ satisfying
$$a^{ij}(t,x)\zeta^i\zeta^j\geq\tau|\zeta|^2\quad\text{and}\quad|a^{ij}|,|b^i|,|c|\leq\tau^{-1},$$
with $(t,x)\in\mathbb{R}^{d+1}$ and $\zeta\in\mathbb{R}^{d}$.
\end{assumption}
\begin{assumption}\label{assump2.2}
There exists a constant $\zeta_{0}\in(0,1]$ satisfying
$$\sup\limits_{i,j}\fint_{Q_R(t^{\ast},x^{\ast})}|a^{ij}-\overline{a}^{ij}|\leq\vartheta ,$$
where
$\overline{a}^{ij}=\displaystyle\fint_{Q_R(t^{\ast}, x^{\ast})}a^{ij}\:dxdt$ for $(t^{\ast},x^{\ast})\in\mathbb{R}^{d+1}$ and $0<R\leq \zeta_{0}$.
\end{assumption}
\begin{assumption}\label{assump2.3}
If $x^{\ast}\in\mathbb{R}^{d}_+$,
based on Assumption \ref{assump2.2}, we need
$$\sup\limits_{i,j}\fint_{Q_R(t^{\ast},x^{\ast})\cap\mathfrak{R}^d_{+}}|a^{ij}-\overline{a}^{ij}|\:dxdt\leq 4\vartheta,$$
where
$\overline{a}^{ij}=\displaystyle\fint_{Q_R(t^{\ast},x^{\ast})\cap\mathfrak{R}^d_{+}}a^{ij}\:dxdt$.
\end{assumption}

In the sequel, we introduce the main results of this paper.
\begin{theorem}\label{the2.1}
Let $p, q\in ( 1, \infty )$, and $[w]_{p,q}\leq \varrho$ for $\varrho\in [ 1, \infty )$. Under Assumption \ref{assump2.3}, for $\vartheta= \vartheta(d, \tau, \alpha, p, q, \varrho) \in (0, 1)$ and  $f \in L_{p,q,w}\left(\mathfrak{R}^{d}_{+}\right)$, then there exists a unique solution $u \in\widetilde{\mathbb{Z}}_{p, q, w, 0}^{\alpha , 2}\left ( \mathfrak{R}^d_{+}\right )$ satisfying
\begin{equation}\label{e2.3}
 \left\{
\begin{aligned}
&-\partial_t^\alpha u+Lu=f(t,x), &&
\text{in }\mathfrak{R}^{d}_{+}, \\
&u=0, &&
\text{on }(0,T)\times \partial\mathbb{R}^{d}_+.
\end{aligned}\right.
\end{equation}
where $$
 Lu=a^{ij}(t,x)D_{ij}u+b^i(t,x)D_iu+c(t,x).$$
Furthermore, we get
\begin{equation}\label{e2.4}
\|\partial_t^\alpha u\|_{p,q,w}+\|u\|_{p,q,w}+\|Du\|_{p,q,w}+\|D^2u\|_{p,q,w}\lesssim\|f\|_{p,q,w},
\end{equation}
where the hidden constant depends on $d,\tau,\alpha,p,q,\varrho,\zeta_{0},T$ and $\|\cdot\|_{p, q, w}=\|\cdot\|_{L_{p,q,w}(\mathfrak{R}^d_{+})}$.
\end{theorem}

\begin{theorem}\label{the2.2}
Let $p, q \in(1, \infty)$, and $[w]_{p,q}\leq \varrho$ for $\varrho\in [ 1, \infty )$. Under Assumption \ref{assump2.2}, for $\vartheta= \vartheta(d, \tau, \alpha, p, q, \varrho) \in (0, 1)$ and  $f_i,g \in L_{p,q,w}\left(\mathfrak{R}^{d}\right)$, $i=1,\ldots,d$, then there exists a unique solution $u \in\mathcal{Z}_{p, q, w, 0}^{\alpha, 1}\left (\mathfrak{R}^d\right )$ satisfying
\begin{equation}\label{e2.5}
-\partial_{t}^{\alpha} u+\mathfrak{L}u=D_{i} f_{i}+g \quad \text { in } \quad \mathfrak{R}^{d},
\end{equation}
where $$\mathfrak{L}u=D_{i}\left(a^{i j} D_{j} u+a^{i} u\right)+b^{i} D_{i} u+c u.$$
Moreover, we obtain
\begin{equation}\label{e2.6}
\left\|\partial_{t}^{\alpha} u\right\|_{\mathbb{Z}_{p, q, w}^{-1}}+\|u\|_{p, q, w}+\|D u\|_{p, q, w} \lesssim\|f\|_{p, q, w}+\|g\|_{p, q, w}
\end{equation}
where the hidden constant depends on $d,\tau,\alpha,p,q,\varrho,\zeta_{0},T$ and $\|\cdot\|_{p, q, w}=\|\cdot\|_{L_{p, q, w}\left(\mathfrak{R}^{d}\right)}$.
\end{theorem}

\begin{theorem}\label{the2.3}
Let $p, q \in(1, \infty)$, and $[w]_{p,q}\leq \varrho$ for $\varrho\in [ 1, \infty )$. Under Assumption \ref{assump2.3}, for $\vartheta= \vartheta(d, \tau, \alpha, p, q, \varrho) \in (0, 1)$ and  $f_i,g \in L_{p,q,w}\left(\mathfrak{R}^{d}_+\right)$, $i=1,\ldots,d$, then there exists a unique solution $u \in\mathcal{Z}_{p, q, w, 0}^{\alpha, 1}\left (\mathfrak{R}^d_+\right )$ satisfying
\begin{equation}\label{e2.7}
 \left\{
\begin{aligned}
&-\partial_t^\alpha u+D_{i}\left(a^{i j} D_{j} u+a^{i} u\right)+b^{i} D_{i} u+c u=D_{i} f_{i}+g , &&
\text{in }\mathfrak{R}^{d}_{+}, \\
&-a^{1j}D_{j}u+a^1u=f_1, &&
\text{on }(0,T)\times \partial\mathbb{R}^{d}_+.
\end{aligned}\right.
\end{equation}
Furthermore, \eqref{e2.6} also holds true in the half space.
\end{theorem}

Next, we introduce a crucial conclusion to  give a control over the mean oscillation under Assumption \ref{assump2.2}.
\begin{lemma} [\cite{dong2}] \label{lem2.1}
If coefficient $a$ satisfies Assumption \ref{assump2.2}, $h\geq R^{\frac{2}{\alpha}}$,and $R\leq2^{-\frac{2}{\alpha}}\zeta_0\leq \zeta_0$. Then
$$\int_{(t^\ast-h,t^\ast)}\fint_{B_r(x^\ast)}|a-\overline{a}|\:dxdt<NhR^{-\frac{2}{\alpha}}\vartheta,$$
where $\overline {a}= \fint_{Q_{r}( t^\ast, x^\ast) }a \:dxdt$ and $N= N( \alpha, d).$
\end{lemma}

\section{Non-divergence form in the half space }\label{sec3}
In this section, we consider non-divergence form
equations in the whole space and the half space. Then, we give the proof for the Theorem \ref{the2.1}. 
\begin{lemma} \label{lem3.1}
Let $\widehat{p} \in(1, \infty)$ and $a^{ij}$ be constant. If $u \in \mathbb{Z}_{\widehat{p}, 0,loc}^{\alpha, 2}\left( \mathfrak{R}^d\right)$  satisfies 
$$\partial_t^\alpha u-a^{ij}D_{ij}u=f\quad\text{in}\quad\mathfrak{R}^d.$$
Then, there exists $\rho \in(0,\frac{1}{4})$ satisfying
\begin{equation}\label{e3.1}
\begin{aligned}
(|D^{2}u-(D^{2}u)_{Q_{\rho R}(t^{\ast},x^{\ast})}|)_{Q_{\rho R}(t^{\ast},x^{\ast})}
\lesssim&\rho^\beta\sum_{k=0}^{\infty}2^{-k\alpha}\left(|D^2 u|^{\widehat{p}}\right)_{(\widetilde{t},t^{\ast}) \times B_{\frac{R}{2}}(x^{\ast})}^{\frac{1}{\widehat{p}}}\\
&+(\rho^\beta+\rho^{\frac{-(d+\frac{2}{\alpha})}{\widehat{p}}})\sum_{k=0}^\infty C_k\left(|f|^{\widehat{p}}\right)_{(\overline{t},t^{\ast})\times B_{R}(x^{\ast})}^{\frac{1}{\widehat{p}}},
\end{aligned}
\end{equation}
with $\left(t^{\ast}, x^{\ast}\right) \in (0, T] \times \mathbb{R}^{d}$, where $\widetilde{t}=t^{\ast}-2^{k}(\frac{R}{2})^{\frac{2}{\alpha}}$, $\overline{t}=t^{\ast}-(2^{k+2}-2)R^{\frac{2}{\alpha}}$ and $\{C_k\}$ satisfies $\sum\limits_{k=0}^{\infty}c_{k}\leq N(d,\tau,\alpha,p)$. The hidden constant is only dependent of $d, \tau, \alpha,\widehat{p}$.
\end{lemma}
\begin{proof}
Based on the result of \cite{dong1}, we get \eqref{e3.1} for functions $a^{ij}(t)$ related to $t$. Hence, when $a^{ij}$ is constant, we have the same result. The proof is completed.
\end{proof}

\begin{lemma} \label{lem3.2}
Let $ \beta > 0$, $R\in (0, \infty )$, $\rho \in (0, \frac{1}{4})$, $p, q\in$ $( 1, \infty )$, $\varrho\in$ $[1, \infty )$ and $[w]_{p,q,\mathbb{R}^{d}}\leq \varrho$.
There is $ \widehat{p} \in (1, \infty )$ and $\lambda \in (1, \infty )$ satisfying
$\widehat{p}<\widehat{p}\lambda<\min\{p,q\}$. Under Assumption \ref{assump2.2}, if $u\in \widetilde{\mathbb{Z}} _{p, q, w, 0}^{\alpha, 2}\left (\mathfrak{R}^d\right )$ has compact support in $(0, T) \times B_{\zeta_{0}}$ and satisfies
\begin{equation}\label{e3.2}
-\partial_t^\alpha u+a^{ij}(t,x)D_{ij}u=f
\end{equation}
in $\mathfrak{R}^d$, then  there holds
\begin{equation}\label{e3.3}
\begin{aligned}
&(|D^{2}u-(D^{2}u)_{(t^{\ast}-(\rho R)^{\frac{2}{\alpha}},t^{\ast})\times B_{\rho R}(x^{\ast})}|)_{(t^{\ast}-(\rho R)^{\frac{2}{\alpha}},t^{\ast})\times B_{\rho R}(x^{\ast})} \\
\lesssim & \rho^{-\frac{d}{\widehat{q}}}(|D^{2}u|^{\widehat{p}})_{(t^{\ast}-(\rho R)^{\frac{2}{\alpha}},t^{\ast})\times B_{\rho R}(x^{\ast})}^{\frac{1}{\widehat{p}}}+
\rho^\beta\sum_{k=0}^{\infty}2^{-k\alpha}\left(|D^2u|^{\widehat{p}}\right)_{(\widetilde{t},t^{\ast}) \times B_{\frac{R}{2}}(x^{\ast})}^{\frac{1}{\widehat{p}}}\\
&+(\rho^\beta+\rho^{\frac{-(d+\frac{2}{\alpha})}{\widehat{p}}})\sum_{k=0}^{\infty}C_{k}\left(|f|^{\widehat{p}}\right)_{(\overline{t},t^{\ast})\times B_{R}(x^{\ast})}^{\frac{1}{\widehat{p}}}\\
\quad&+(\rho^\beta+\rho^{\frac{-(d+\frac{2}{\alpha})}{\widehat{p}}})\vartheta^{\frac{1}{\gamma\widehat{p}}}\sum_{k=0}^{\infty}C_{k}2^{\frac{k+2}{\gamma\widehat{p}}}\left(|D^{2}u|^{\lambda \widehat{p}}\right)_{(\overline{t},t^{\ast})\times B_{R}(x^{\ast})}^{\frac{1}{\lambda \widehat{p}}},
\end{aligned}
\end{equation}
for $( t^{\ast}, x^{\ast}) \in (0, T] \times \mathbb{R}^{d}$, where $\gamma =\frac{\lambda}{\lambda-1}$, $\widehat{q}=\frac{\widehat{p}}{\widehat{p}-1}$, and the hidden constant is dependent of $d, \tau, \alpha, p, q$ and $\varrho$. And for $t \leq 0$, both $u$ and $f$ are equal to $0$.
\end{lemma}
\begin{proof}
 Set $\beta_1,\beta_2>0$ satisfying $p-\beta_1>1,q-\beta_2>1 $. In view of $w_1\in A_p(\mathbb{R}^{d},dt)$, $w_2\in A_q(\mathbb{R}^{d},dx)$, we get $w_1\in A_{p-\beta_1}(\mathbb{R}^{d},dt)$ $w_2\in A_{q-\beta_2}(\mathbb{R}^{d},dx)$. We choose $\widehat{p},\lambda\in(1,\infty)$ satisfying $\widehat{p}\lambda=\min\left\{\frac{p}{p-\beta_1},\frac{q}{q-\beta_2}\right\}>1$. Hence, we obtain $p-\beta_1<\frac{p}{\widehat{p}\lambda}<\frac{p}{\widehat{p}}$ and $q-\beta_2<\frac{q}{\widehat{p}\lambda}<\frac{q}{\widehat{p}}$.
By the nature of $A_p$ weight, it yeilds
$$w_1\in A_{p-\beta_1}\subset A_{\frac{p}{\widehat{p}\lambda}}\subset A_{\frac{p}{\widehat{p}}}(\mathbb{R}^{d},dt)\quad \text{and}\quad
w_2\in A_{q-\beta_2}\subset A_{\frac{q}{\widehat{p}\lambda}}\subset A_{\frac{q}{\widehat{p}}}(\mathbb{R}^{d},dx).$$  
By \cite[Lemma 5.10]{dong}, it follows that $u\in\mathbb{Z}^{\alpha,2}_{\widehat{p}\lambda,0,\text{loc}}\big( \mathfrak{R}^d\big)$  in view of  $u\in\mathbb{Z}_{p,q,w,0}^{\alpha,2}\left( \mathfrak{R}^d\right)$.

If $R\geq 2^{-\frac{\alpha}{2}}\zeta_{0}$, by H\"{o}lder's inequality, we obtain
\begin{equation}\label{e3.4}
\begin{aligned}
&(|D^{2}u-(D^{2}u)_{(t^{\ast}-(\rho R)^{\frac{2}{\alpha}},t^{\ast})\times B_{\rho R}(x^{\ast})}|)_{(t^{\ast}-(\rho R)^{\frac{2}{\alpha}},t^{\ast})\times B_{\rho R}(x^{\ast})}\\
\leq&2(|D^{2}u|)_{(t^{\ast}-(\rho R)^{\frac{2}{\alpha}},t^{\ast})\times B_{\rho R}(x^{\ast})}\\
\leq&2(|D^{2}u|^{\widehat{p}})_{(t^{\ast}-(\rho R)^{\frac{2}{\alpha}},t^{\ast})\times B_{\rho R}(x^{\ast})}^{\frac{1}{\widehat{p}}}\left(\frac{|B_{\zeta_{0}}|}{|B_{\rho R}|}\right)^{\frac{1}{\widehat{q}}}\\
\lesssim&\rho^{-\frac{d}{\widehat{q}}}(|D^{2}u|^{\widehat{p}})_{(t^{\ast}-(\rho R)^{\frac{2}{\alpha}},t^{\ast})\times B_{\rho R}(x^{\ast})}^{\frac{1}{\widehat{p}}}.
\end{aligned}
\end{equation}
For $R< 2^{-\frac{\alpha}{2}}\zeta_{0}$, let
$$-\partial_t^\alpha u+\overline{a}^{ij}D_{ij}u=\overline{f},$$
where $\overline{f}:=f+(\overline{a}^{ij}-a^{ij})D_{ij}u$ and
\begin{equation*}
\overline{a}^{ij}=\fint_{Q_{2R}(t^{\ast},x^{\ast})}a^{ij} dtdz.
\end{equation*}
From Lemma \ref{lem3.1}, we get
\begin{equation*}
\begin{aligned}
(|D^{2}u-(D^{2}u)_{Q_{\rho R}(t^{\ast},x^{\ast})}|)_{Q_{\rho R}(t^{\ast},x^{\ast})}
\lesssim&\rho^\beta\sum_{k=0}^{\infty}2^{-k\alpha}\left(|D^2 u|^{\widehat{p}}\right)_{(\widetilde{t},t^{\ast}) \times B_{\frac{R}{2}}(x^{\ast})}\\
&+(\rho^\beta+\rho^{\frac{-(d+\frac{2}{\alpha})}{\widehat{p}}})\sum_{k=0}^\infty C_k\left(|\overline{f}|^{\widehat{p}}\right)_{(\overline{t},t^{\ast})\times B_{R}(x^{\ast})}^{\frac{1}{\widehat{p}}}.
\end{aligned}
\end{equation*}
Then, by the H\"{o}lder inequality, we deduce
\begin{equation}\label{e3.5}
\begin{aligned}
\left(|\overline{f}|^{\widehat{p}}\right)_{\left(\overline{t},t^{\ast}\right)\times B_{R}(x^{\ast})}^{\frac{1}{\widehat{p}}}
\leq&\left(|f|^{\widehat{p}}\right)_{\left(\overline{t},t^{\ast}\right)\times B_{R}(x^{\ast})}^{\frac{1}{\widehat{p}}}\\
&+\left(|\overline{a}^{ij}-a^{ij}|^{\gamma \widehat{p}}\right)_{(\overline{t},t^{\ast})\times B_{R}(x^{\ast})}^{\frac{1}{\gamma\widehat{p}}}
 \left(|D^{2}\widehat{u}|^{\lambda \widehat{p}}\right)_{(\overline{t},t^{\ast})\times B_{R}(x^{\ast})}^{\frac{1}{\lambda \widehat{p}}}.
\end{aligned}
\end{equation}
Furthermore, it follows from Lemma \ref{lem2.1} that
\begin{align}\label{e3.6}
\left(|\overline{a}^{ij}-a^{ij}|\right)_{(\overline{t},t^{\ast})\times B_{R}(x^{\ast})}\lesssim (2^{k+2}-2)\vartheta. 
\end{align}
Combining \eqref{e3.5} with \eqref{e3.6}, we obtain
\begin{equation}\label{e3.7}
\begin{aligned}
&\sum_{k=0}^{\infty}C_{k}\left(|\overline{f}|^{\widehat{p}}\right)_{(\overline{t},t^{\ast})\times B_{R}(x^{\ast})}^{\frac{1}{\widehat{p}}}\\
\lesssim&\sum_{k=0}^{\infty}C_{k}\left(|f|^{\widehat{p}}\right)_{(\overline{t},t^{\ast})\times B_{R}(x^{\ast})}^{\frac{1}{\widehat{p}}}+ \vartheta^{\frac{1}{\gamma\widehat{p}}}\sum_{k=0}^{\infty}C_{k}2^{\frac{k+2}{\gamma \widehat{p}}}\left(|D^{2}u|^{\lambda \widehat{p}}\right)_{(\overline{t},t^{\ast})\times B_{R}(x^{\ast})}^{\frac{1}{\lambda \widehat{p}}}.
\end{aligned}
\end{equation}
Hence, from \eqref{e3.4} and \eqref{e3.7}, we deduce \eqref{e3.3}. The proof is completed.
\end{proof}

For each integer $m \in \mathbb{Z}$, choose an integer $k(m)$ satisfying
$$k(m) \leq \frac{2m}{\alpha}<k(m)+1.$$
Therefore, we get
$$
\mathbb{C}_{m}:=\left\{Q_{\vec{j}}^{m}=Q_{\left(j_{0}, \ldots, j_{d}\right)}^{m}: \vec{j}=\left(j_{0}, \ldots, j_{d}\right) \in \mathbb{Z}^{d+1}\right\},$$
 with $$
Q_{\vec{j}}^{m}=\left[\frac{j_{0}}{2^{k(m)}}+T, \frac{j_{0}+1}{2^{k(m)}}+T\right) \times\left[\frac{j_{1}}{2^{m}}, \frac{j_{1}+1}{2^{m}}\right) \times \cdots \times\left[\frac{j_{d}}{2^{m}}, \frac{j_{d}+1}{2^{m}}\right).$$
Moreover, we introduce the dyadic sharp function of  $h$  by
$$h_{dz}^{\#}(t^{\ast}, x^{\ast})=\sup _{m<\infty} \fint_{(t^{\ast}, x^{\ast}) \in Q_{\vec{j}}^{m}}\left|h(s, z)-h_{\mid m}(t^{\ast}, x^{\ast})\right| d z d s,$$
where
$$h_{\mid m}(t^{\ast}, x^{\ast})=\fint_{Q_{\vec{j}}^{m}} h(s, z) dzds.$$

\begin{lemma}\label{lem3.3}
Let $p, q\in$ $( 1, \infty )$, and $[w]_{p,q,\mathbb{R}^{d}}\leq \varrho$ for  $\varrho\in$ $[ 1, \infty )$. Under Assumption \ref{assump2.2}, for any $u\in \widetilde{\mathbb{Z}} _{p, q, w, 0}^{\alpha, 2}\left(\mathfrak{R}^d\right )$ with compact support in $(0, T) \times B_{\zeta_{0}}$ satisfying \eqref{e3.2} in $\mathfrak{R}^d$,  there holds
\begin{equation}\label{e3.8}
\|\partial_t^\alpha u\|_{L_{p,q,w}(\mathfrak{R}^d)}+\|D^2u\|_{L_{p,q,w}(\mathfrak{R}^d)}\lesssim \|f\|_{L_{p,q,w}(\mathfrak{R}^d)},
\end{equation}
where the hidden constant is dependent of $d, \tau, \alpha, p, q$ and $\varrho$.
\end{lemma}
\begin{proof}
For any $(t^{\ast}, x^{\ast}) \in \mathfrak{R}^d$, let $t^{\prime}=\min(T,t^{\ast}+\frac{\widehat{R}}{2})\in(-\infty,T]$ with
 $\widehat{R}=R^{\frac{2}{\alpha}}$,
$Q_{\vec{j}}^m\subset(t^{\prime}-\widehat{R},t^{\prime})\times B_{R}(x^{\ast})$, and $|(t^{\prime}-\widehat{R},t^{\prime})\times B_{R}(x^{\ast})|\lesssim |Q_{\vec{j}}^m|$.

Hence, by \eqref{e3.3}, we obtain
\begin{equation*}
\begin{aligned}
&\fint_{Q_{\vec{j}}^{n}\ni(t^{\ast},x^{\ast})}|D^2u(s,z)-(D^2u)_{|m}(t^{\ast},x^{\ast})| dz ds\\ 
\leq&(|D^2u-(D^2u)_{(t^{\prime}-\widehat{R},t^{\prime})\times B_{R}(x^{\ast})}|)_{(t^{\prime}-\widehat{R},t^{\prime})\times B_{R}(x^{\ast})}\\
\lesssim&\rho^{-\frac{d}{\widehat{q}}}(|D^{2}u|^{\widehat{p}})_{(t^{\prime}-\widehat{R},t^{\ast})\times B_{R}(x^{\ast})}^{\frac{1}{\widehat{p}}}+\rho^\beta\sum_{k=0}^{\infty}2^{-k\alpha}\left(|D^2u|^{\widehat{p}}\right)^{\frac{1}{\widehat{p}}}_{\left(t^{\prime}-2^{k}(\frac{R}{2\rho})^{\frac{2}{\alpha}},t^{\prime}\right) \times B_{\frac{R}{2\rho}}(x^{\ast})}\\
&+(\rho^\beta+\rho^{\frac{-(d+\frac{2}{\alpha})}{\widehat{p}}})\sum_{k=0}^{\infty}C_{k}\left(|f|^{\widehat{p}}\right)_{\left(t^{\prime}-(2^{k+2}-2)\left({\frac{R}{\rho}}\right)^{\frac{2}{\alpha}}, t^{\prime}\right)\times B_{\frac{R}{\rho}}(x^{\ast})}^{\frac{1}{\widehat{p}}}\\
&+(\rho^\beta+\rho^{\frac{-(d+\frac{2}{\alpha})}{\widehat{p}}})\vartheta^{\frac{1}{\gamma\widehat{p}}}\sum_{k=0}^{\infty}C_{k}2^{\frac{k+2}{\gamma\widehat{p}}}\left(|D^{2}u|^{\lambda \widehat{p}}\right)_{\left(t^{\prime}-(2^{k+2}-2)\left({\frac{R}{\rho}}\right)^{\frac{2}{\alpha}}, t^{\prime}\right)\times B_{\frac{R}{\rho}}(x^{\ast})}^{\frac{1}{\lambda \widehat{p}}}.
\end{aligned}
\end{equation*}
Furthermore, it follows from Lemma \ref{lem3.2} and the definition of the dyadic sharp function  that
\begin{equation*}
\begin{aligned}
&(D^2u)_{dz}^\#(t^{\ast},x^{\ast})\\
\lesssim &\left(\rho^{-\frac{d}{\widehat{q}}}+\rho^{\beta}\right)(\widehat{\mathcal{SM}}|D^{2}u|^{\widehat{p}})^{\frac{1}{\widehat{p}}}(t^{\ast},x^{\ast})+\left(\rho^{\beta}+\rho^{\frac{-(d+\frac{2}{\alpha})}{\widehat{p}}}\right)\vartheta ^\frac{1}{\gamma\widehat{p}}(\widehat{\mathcal{SM}}|D^{2}u|^{\lambda \widehat{p}})^{\frac{1}{\lambda\widehat{p}}}(t^{\ast},x^{\ast})\\
\quad&+\left(\rho^{\beta}+\rho^{\frac{-(d+\frac{2}{\alpha})}{\widehat{p}}}\right)(\widehat{\mathcal{SM}}|f|^{\widehat{p}})^{\frac{1}{\widehat{p}}}(t^{\ast},x^{\ast}).
\end{aligned}
\end{equation*}
Therefore, based on the weighted sharp function theorem, i.e, $\|h\|_{p,q,w}\lesssim \|h_{{\mathrm{dz}}}^{\#}\|_{p,q,w},$
and the weighted maximal function theorem for strong maximal functions, i.e, $\|\mathcal{SM}h\|_{p,q,w}\lesssim \|h\|_{p,q,w},$
it is clear that there exists a sufficiently small $\rho \in (0,\frac{1}{4})$ satisfying
\begin{equation*}
\begin{aligned}
\|D^{2}u\|_{p,q,w}&\lesssim (\rho^{-\frac{d}{\widehat{q}}}+2\rho^\beta+\rho^{\frac{-(d+\frac{2}{\alpha})}{\widehat{p}}}\vartheta^\frac{1}{\gamma\widehat{p}})\|D^2u\|_{p,q,w}
 +(\rho^\beta+\rho^{\frac{-(d+\frac{2}{\alpha})}{\widehat{p}}})\|f\|_{p,q,w}
\\&\lesssim \frac{1}{2}\|D^{2}u\|_{p,q,w}+\|f\|_{p,q,w}.
\end{aligned}
\end{equation*}

Eventually, we obtain
\begin{equation}\label{e3.9}
\|D^2u\|_{p,q,w}\lesssim \|f\|_{p,q,w}.
\end{equation}
Hence, by the fact
$-\partial_t^\alpha u=-a^{ij}(t,x)D_{ij}u+f$ in $\mathfrak{R}^d$,
we deduce \eqref{e3.8} from \eqref{e3.9}. The proof is completed.
\end{proof}

\begin{lemma}\label{lem3.4}
Let $p\in$ $(1, \infty )$,  $[w]_{p,q,\mathbb{R}^{d}}\leq \varrho$ for $\varrho\in$ $[ 1, \infty )$. Under Assumption \ref{assump2.2}, for any
$u\in \widetilde{\mathbb{Z}}_{p, q, w, 0}^{\alpha, 2}\left (\mathfrak{R}^d\right )$ satisfying
\begin{equation*}
-\partial_t^\alpha u+a^{ij}D_{ij}u+b^iD_{i}u+cu=f \qquad\text{in}\quad\mathfrak{R}^d,
\end{equation*}
there holds
\begin{align}\label{e3.10}
\|\partial_t^\alpha u\|_{p,q,w}+\|D^2u\|_{p,q,w}\lesssim \|f\|_{p,q,w}+\|u\|_{p,q,w}.
\end{align}
Furthermore, we deduce
\begin{align}\label{e3.11}
 \|I^{\alpha}_tf\|_{ p,q,w}\lesssim T^{\alpha}\|f\|_{ p,q,w}
\end{align}
and 
\begin{align}\label{e3.12}
\|u\|_{p,q,w}\lesssim T^\alpha\|\partial_t^\alpha u\|_{p,q,w},
\end{align}
where the hidden constant is dependent of $d, \tau, \alpha, p, q, \varrho, \zeta_{0}$ and $\|\cdot\|_{p, q, w}=\|\cdot\|_{L_{p,q,w}(\mathfrak{R}^d)}$.
\end{lemma}
\begin{proof}
First, we consider $p=q$ and $b^i=c=0$. By using a partition of unity argument in the spatial, we set $u_l(t,x)=u(t,x)\phi_l(x)$, where $\phi_l(x) \geq 0$, $\phi_l(x) \in C_0^\infty (\mathbb{R}^{d})$, $supp(\phi_l(x)) \subset B_{\frac{\zeta_{0}}{\sqrt{2}}}(x_l)$ and
$$1\leq\sum_{l=1}^{\infty}\left|\phi_{l}(x)\right|^{p} \leq C_{0}, \quad \sum_{l=1}^{\infty}\left| D_{x} \phi_{l}(x)\right|^{p} \leq C_{1},$$
where $C_{0}$ is dependent of $p$ and $C_{1}$ is dependent of $d, \alpha, p, \zeta_{0}$. Then, $u_l$ satisfies
\begin{equation*}
\begin{aligned}
-\partial_{t}^{\alpha} u_{l} +a^{i j} D_{i,j} u_{l}& =-\left(\partial_{t}^{\alpha} u\right) \phi_{l}+\left(a^{i j} D_{i,j}u\right)\phi_{l}+a^{ij} D_{i}u D_{j} \phi_{l}+2a^{ij}uD_{i,j}\phi_{l}\\
& =f\phi_{l}+2a^{ij} D_{i}u D_{j} \phi_{l}+a^{ij}uD_{i,j}\phi_{l}.
\end{aligned}    
\end{equation*}

It follows from Lemma \ref{lem3.3} that
\begin{align*}
\|\partial_t^\alpha u_l\|_{p,w}+\|D^2u_l\|_{p,w}
&\lesssim \|f\|_{p,w}+\|Du\|_{p,w}+\|u\|_{p,w}.
\end{align*}
Combining with the similar proof of \cite[Corollary 4.3]{dong2} for $\alpha \in (1,2)$ and  the interpolation inequality
$$\int_{\mathbb{R}^{d}}|Du|^pwdx\lesssim\rho^p\int_{\mathbb{R}^{d}}\left|D^2u|^pw\right.dx+ \rho^{-p}\int_{\mathbb{R}^{d}}\left|u|^pw\right.dx,$$ it yields \eqref{e3.10}. And the same result is obtained by the extrapolation theorem (see, e.g., \cite[Theorem 2.5]{dong}) for the case of $p\neq q$. 

Note that the $L_{p,q}$-estimates in \cite[Lemma 5.5]{dong1}, we obtain \eqref{e3.11}. Moreover, by the proof of \cite[Lemma 5.6]{dong1} and \eqref{e3.11}, set $u=I^\alpha_t\partial_t^\alpha u$, we deduce
$$\|u\|_{p,q,w}=\|I^\alpha_t\partial_t^\alpha u\|_{p,q,w}\lesssim T^\alpha\|\partial_t^\alpha u\|_{p,q,w}.$$ Therefore, \eqref{e3.12} holds. The proof is completed.
\end{proof}

\begin{lemma} \label{lem3.5}
Let $p, q\in ( 1, \infty )$, and $[w]_{p,q}\leq \varrho$ for $\varrho\in [ 1, \infty )$. Under Assumption \ref{assump2.2}, for $\vartheta \in (0, 1)$ and  $f \in L_{p,q,w}\left(\mathfrak{R}^{d}\right)$, then there is a unique solution $u \in\mathbb{Z}_{p, q, w, 0}^{\alpha , 2}\left ( \mathfrak{R}^d\right )$ satisfying
\begin{equation}\label{e3.13}
\partial_t^\alpha u-a^{ij}(t,x)D_{ij}u-b^i(t,x)D_iu-c(t,x)=f \qquad\text{in}\quad\mathfrak{R}^{d}.
\end{equation}
Furthermore, we get
\begin{equation}\label{e3.14}
\|\partial_t^\alpha u\|_{p,q,w}+\|u\|_{p,q,w}+\|Du\|_{p,q,w}+\|D^2u\|_{p,q,w}\lesssim\|f\|_{p,q,w},
\end{equation}
where the hidden constant depends on $d,\tau,\alpha,p,q,\varrho,\zeta_{0},T$ and $\|\cdot\|_{p, q, w}=\|\cdot\|_{L_{p,q,w}(\mathfrak{R}^d)}$.
\end{lemma}
\begin{proof}
 In order to prove \eqref{e3.14}, by Lemma \ref{lem3.4}, we need to verify
\begin{equation}\label{e3.15}
\|u\|_{p,q,w}\lesssim \|f\|_{p,q,w}.
\end{equation} 
Consider a positive integer $n$ and an integer $l$, set $s_{n,l}=\frac{lT}{n}$ with $-1\leq l \leq n-1$, one can choose smooth functions $\varphi_l(t)$ satisfying $\varphi_l(t)=0$ if $t\leq s_{n,l-1}$ and $\varphi_l(t)=1$ if $t\geq s_{n,l}$ with $|\varphi_l^{\prime}|\leq\frac{2n}{T}$ for $ 0\leq l\leq n-1$.

By Lemma \ref{lem3.2}, we obtain $u\varphi_l\in\mathbb{Z}_{p,q,w,0}^{\alpha,2}((s_{n,l-1},s_{n,l+1})\times\mathbb{R}^{d})$ satisfying
\begin{equation}\label{e3.16}
-\partial_{t}^{\alpha}(u\varphi_l)+a^{ij}D_{ij}(u\varphi_l)+b^{i}D_{i}(u\varphi_l)+c(u\varphi_l)=f\varphi_l+\overline{g}_{l},
\end{equation}
where for $ s_{n,l-1}\leq t\leq s_{n,l+1}$ and
\begin{align*}
\overline{g}_{l}(t,x)
&=\frac{\alpha}{\Gamma(1-\alpha)}\int_{-\infty}^t(t-s)^{-\alpha-1}\left(\varphi_l(t)-\varphi_l(s)\right)u(s,x)\chi_{s\leq s_{n,l}}ds.
\end{align*}
Then, for $l=0$, we have $\overline{g}_0(t,x)=0$. For $1\leq l \leq n-1$, it follows from $|\varphi_l^{\prime}|\leq\frac{2n}{T}$ that
\begin{equation}\label{e3.17}
\begin{aligned}
\overline{g}_{l}(t,x)
&\leq\frac{2n}{T}\frac{\alpha}{\Gamma(1-\alpha)}\int_0^{t}(t-s)^{-\alpha}|u(s,x)|\chi_{s\leq s_{n,l}}ds.
\end{aligned}
\end{equation}
Hence, together with \eqref{e3.11} and \eqref{e3.17}, we deduce
\begin{align*}
\|\overline{g}_{l}\|_{p,q,w((s_{n,l-1},s_{n,l+1})\times\mathbb{R}^{d})}
&\leq \frac{2n\alpha}{T}\|I_{0}^{1-\alpha}u(t,x)\chi_{t\leq s_{n,l+1}}\|_{p,q,w((0,s_{n,l+1})\times\mathbb{R}^{d})}\\
&\leq\frac {2n\alpha}{T}\left(s_{n,l+1}\right)^{1-\alpha}\|u(t,x)\chi_{t\leq s_{n,l}}\|_{p,q,w((0,s_{n,l+1})\times\mathbb{R}^{d})}\\
&\leq T^{-\alpha}\|u\|_{p,q,w((0,s_{n,l})\times\mathbb{R}^{d})},
\end{align*}
where $s_{n,l+1}=\frac{(l+1)T}n$ and $1\leq l\leq n-1$. 

By \eqref{e3.10}, \eqref{e3.12}, \eqref{e3.13} and the last inequality, for $0 \leq l \leq n-1$, it yields
\begin{align*}
&\|u\|_{p,q,w((s_{n,l},s_{n,l+1})\times\mathbb{R}^{d})}\\
\leq&\|u\varphi_l\|_{p,q,w((s_{n,l-1},s_{n,l+1})\times\mathbb{R}^{d})}\\
\lesssim& \left(\frac{T}{n}\right)^{\alpha}\|\partial_{t}^{\alpha}(u\varphi_l)\|_{p,q,w((s_{n,l-1},s_{n,l+1})\times\mathbb{R}^{d})} \\
\lesssim &\|f\varphi_l\|_{p,q,w((s_{n,l-1},s_{n,l+1})\times\mathbb{R}^{d})}+n^{-\alpha}\|u\|_{p,q,w((0,s_{n,l})\times\mathbb{R}^{d})} +\left(\frac{T}{n}\right)^{\alpha}\|u\|_{p,q,w((s_{n,l-1},s_{n,l+1})\times\mathbb{R}^{d})}\\
\lesssim&\|f\|_{p,q,w((0,s_{n,l+1})\times\mathbb{R}^{d})}+\|u\|_{p,q,w((0,s_{n,l})\times\mathbb{R}^{d}))},
\end{align*}
with $n$ large enough. Therefore, by induction we derive \eqref{e2.4} from $\|u\|_{p,q,w((0,s_0)\times\mathbb{R}^{d})}=0$.

According to the result of \cite{han}, we get the existence of solutions for \eqref{e3.13}. The proof is completed.
\end{proof}

In the half space, let $\mathfrak{R}^d_{+} :=(0, T) \times \mathbb{R}^{d}_{+}$, $B_R^+(x^{\ast}):=B_R(x^{\ast})\cap\mathbb{R}^{d}_{+}$ and
$$
(\widehat{\mathcal{SM}}f)(t^{\ast}, x^{\ast})=\sup_{\begin{array}{c}(t^{\ast}, x^{\ast}) \in Q_{R_{1},R_{2}}(s,z)\cap\mathfrak{R}^d_{+}\\(s,z)\in\mathfrak{R}^d_{+}
\end{array}}\fint_{Q_{R_{1},R_{2}}(s,z)\cap\mathfrak{R}^d_{+}}|f(u,y)|\:dy\:du.
$$
Then, we define
$\widehat{w}(t,x)=w_{1}(t)\widehat{w_{2}}(x)$, for $\widehat{w_{2}}(x)=w_{2}(|x^{1}|,x^{\prime})$, that is, $\widehat{w_{2}}(x) \in A_{q}(\mathbb{R}^{d})$ is the even extension of $w_2\in A_{q}(\mathbb{R}^{d}_{+})$ in the $x^1$ direction.


\begin{lemma}\label{lem3.6}
Set $\widehat u(t,x):=u(t,|x^1|,x^{\prime})\mathrm{sgn}x^1$ for $u\in \widetilde{\mathbb{Z}}_{p, q, w, 0}^{\alpha, 2}(\mathfrak{R}^d_{+})$. Then $\widehat{u} \in\mathbb{Z}_{p,q,\widehat{w},0}^{\alpha,2}(\mathfrak{R}^d)$ and
\begin{equation}\label{e3.18}
\|u\|_{\mathbb{Z}_{p,q,w}^{\alpha,2}(\mathfrak{R}^d_{+})}\leq\|\widehat{u}\|_{\mathbb{Z}_{p,q,\widehat{w}}^{\alpha,2}(\mathfrak{R}^d)}\leq2\|u\|_{\mathbb{Z}_{p,q,w}^{\alpha,2}(\mathfrak{R}^d_{+})}.
\end{equation}
\end{lemma}
\begin{proof}
According to the result of \cite[Lemma 8.2.1]{kry} for $u \in W^{k}_{p}(\mathfrak{R}^d_{+})$,   \eqref{e3.18} follows on $w =1$. Combining with the definition of $\widehat u(t,x)$, we obtain \eqref{e3.18}. The proof is completed.
\end{proof}

\begin{lemma} \label{lem3.7}
Let $\widehat{p} \in(1, \infty)$ and $a^{ij}$ be constant. If $u \in \widetilde{\mathbb{Z}}_{\widehat{p}, 0,loc}^{\alpha, 2}\left( \mathfrak{R}^d\right)$  satisfies 
$$\partial_t^\alpha u-a^{ij}D_{ij}u=f\quad\text{in}\quad\mathfrak{R}^d_+.$$
Then, there exists $\rho \in(0,\frac{1}{4})$ satisfying
\begin{equation}\label{e3.19}
\begin{aligned}
&(|D^{2}u-(D^{2}u)_{(t^{\ast}-(\rho R)^{\frac{2}{\alpha}},t^{\ast})\times B^{+}_{\rho R}(x^{\ast})}|)_{(t^{\ast}-(\rho R)^{\frac{2}{\alpha}},t^{\ast})\times B^{+}_{\rho R}(x^{\ast})} \\
\lesssim&\rho^\beta\sum_{k=0}^{\infty}2^{-k\alpha}\left(|D^2 u|^{\widehat{p}}\right)_{(\widetilde{t},t^{\ast}) \times B^+_{\frac{R}{2}}(x^{\ast})}^{\frac{1}{\widehat{p}}}+(\rho^\beta+\rho^{\frac{-(d+\frac{2}{\alpha})}{\widehat{p}}})\sum_{k=0}^\infty C_k\left(|f|^{\widehat{p}}\right)_{(\overline{t},t^{\ast})\times B^+_{R}(x^{\ast})}^{\frac{1}{\widehat{p}}},
\end{aligned}
\end{equation}
with $\left(t^{\ast}, x^{\ast}\right) \in (0, T] \times \mathbb{R}^{d}_+$, where $\widetilde{t}=t^{\ast}-2^{k}(\frac{R}{2})^{\frac{2}{\alpha}}$, $\overline{t}=t^{\ast}-(2^{k+2}-2)R^{\frac{2}{\alpha}}$ and $\{C_k\}$ satisfies $\sum\limits_{k=0}^{\infty}c_{k}\leq N(d,\tau,\alpha,p)$. The hidden constant is only dependent of $d, \tau, \alpha,\widehat{p}$.
\end{lemma}
\begin{proof}
First, define the even and odd extension 
$$\widehat{a}^{11}(t,x)=a^{11}(t,|x^1|,x^{\prime}),\quad \widehat{a}^{1j}(t,x)=a^{1j}(t,|x^1|,x^{\prime})\mathrm{sgn}x^1,j\geq2,$$
$$\widehat{a}^{i1}(t,x)=a^{i1}(t,|x^1|,x^{\prime})\mathrm{sgn}x^1,i\geq2,\quad \widehat{a}^{ij}(t,x)=a^{ij}(t,|x^1|,x^{\prime}),i,j\geq2,$$
$$\widehat{u}(t,x)=u(t, |x^1|,x^{\prime})\mathrm{sgn}x^1,\quad \widehat{f}(t,x)=f(t, |x^1|,x^{\prime})\mathrm{sgn}x^1,$$ 
and $\widehat{u}$ satisfies
$$-\partial_t^\alpha\widehat{u}+\widehat{a}^{ij}D_{ij}\widehat{u}=\widehat{f}\quad\text{in}\quad\mathfrak{R}^d.$$
It follows from Lemma \ref{lem3.1} and Lemma \ref{lem3.6} that
\begin{equation}\label{e3.20}
\begin{aligned}
(|D^{2}\widehat{u}-(D^{2}\widehat{u})_{Q_{\rho R}(t^{\ast},x^{\ast})}|)_{Q_{\rho R}(t^{\ast},x^{\ast})}
\lesssim&\rho^\beta\sum_{k=0}^{\infty}2^{-k\alpha}\left(|D^2 \widehat{u}|^{\widehat{p}}\right)_{(\widetilde{t},t^{\ast}) \times B_{\frac{R}{2}}(x^{\ast})}^{\frac{1}{\widehat{p}}}\\
&+(\rho^\beta+\rho^{\frac{-(d+\frac{2}{\alpha})}{\widehat{p}}})\sum_{k=0}^\infty C_k\left(|\widehat{f}|^{\widehat{p}}\right)_{(\overline{t},t^{\ast})\times B_{R}(x^{\ast})}^{\frac{1}{\widehat{p}}}.
\end{aligned}
\end{equation}
Combining the \eqref{e3.20},
\begin{equation*}
\fint_{B_{R}\left(x^{\ast}\right)}|\widehat{f}| d x \leq \frac{1}{\left|B_{R}^{+}\right|} \fint_{B_{R}\left(x^{\ast}\right)}|\widehat{f}| d x \leq \frac{2}{\left|B_{R}^{+}\right|} \fint_{B_{R}^{+}\left(x^{\ast}\right)}|f| d x    
\end{equation*}
and
\begin{equation}\label{e3.21}
\begin{aligned}
&(|D^{2} u-\left(D^{2} u\right)_{(t^{\ast}-(\rho R)^{\frac{2}{\alpha}},t^{\ast})\times B^{+}_{\rho R}(x^{\ast})}|)_{(t^{\ast}-(\rho R)^{\frac{2}{\alpha}},t^{\ast})\times B^{+}_{\rho R}(x^{\ast})} \\
\leq &4(|D^{2}\widehat{u}-(D^{2}\widehat{u})_{Q_{\rho R}(t^{\ast},x^{\ast})}|)_{Q_{\rho R}(t^{\ast},x^{\ast})},
\end{aligned}
\end{equation}
we deduce \eqref{e3.19}. The proof is completed.
\end{proof}

\begin{lemma}\label{lem3.8}
Let $ \beta > 0$, $R\in (0, \infty )$, $\rho \in (0, \frac{1}{4})$, $p, q\in$ $( 1, \infty )$, $\varrho\in$ $[1, \infty )$ and $[w]_{p,q,\mathbb{R}^{d}_{+}}\leq \varrho$.
There is $ \widehat{p} \in (1, \infty )$ and $\lambda \in (1, \infty )$ satisfying
$\widehat{p}<\widehat{p}\lambda<\min\{p,q\}$. Under Assumption \ref{assump2.3}, if $u\in \widetilde{\mathbb{Z}}_{p, q, w, 0}^{\alpha, 2}\left (\mathfrak{R}^d_{+}\right )$ has compact support in $(0, T) \times B^{+}_{\zeta_{0}}$ and satisfies
\begin{equation}\label{e3.22}
-\partial_t^\alpha u+a^{ij}(t,x)D_{ij}u=f
\end{equation}
in $\mathfrak{R}^d_{+}$, then  there holds
\begin{equation}\label{e3.23}
\begin{aligned}
&(|D^{2}u-(D^{2}u)_{(t^{\ast}-(\rho R)^{\frac{2}{\alpha}},t^{\ast})\times B^{+}_{\rho R}(x^{\ast})}|)_{(t^{\ast}-(\rho R)^{\frac{2}{\alpha}},t^{\ast})\times B^{+}_{\rho R}(x^{\ast})} \\
\lesssim&\rho^{-\frac{d}{\widehat{q}}}(|D^{2}u|^{\widehat{p}})_{(t^{\ast}-(\rho R)^{\frac{2}{\alpha}},t^{\ast})\times B^{+}_{\rho R}(x^{\ast})}^{\frac{1}{\widehat{p}}}+
\rho^\beta\sum_{k=0}^{\infty}2^{-k\alpha}\left(|D^2u|^{\widehat{p}}\right)_{(\widetilde{t},t^{\ast}) \times B^{+}_{\frac{R}{2}}(x^{\ast})}\\
\quad&+(\rho^\beta+\rho^{\frac{-(d+\frac{2}{\alpha})}{\widehat{p}}})\sum_{k=0}^{\infty}C_{k}\left(|f|^{\widehat{p}}\right)_{(\overline{t},t^{\ast})\times B^{+}_{R}(x^{\ast})}^{\frac{1}{\widehat{p}}}\\
\quad&+(\rho^\beta+\rho^{\frac{-(d+\frac{2}{\alpha})}{\widehat{p}}})\vartheta^{\frac{1}{\gamma\widehat{p}}}\sum_{k=0}^{\infty}C_{k}2^{\frac{k+2}{\gamma\widehat{p}}}\left(|D^{2}u|^{\lambda \widehat{p}}\right)_{(\overline{t},t^{\ast})\times B^{+}_{R}(x^{\ast})}^{\frac{1}{\lambda \widehat{p}}},
\end{aligned}
\end{equation}
for $( t^{\ast}, x^{\ast}) \in (0, T] \times \mathbb{R}^{d}_{+}$, where $\gamma =\frac{\lambda}{\lambda-1}$, $\widehat{q}=\frac{\widehat{p}}{\widehat{p}-1}$, and the hidden constant is dependent of $d, \tau, \alpha, p, q$ and $\varrho$. And for $t \leq 0$, both $u$ and $f$ are equal to $0$.
\end{lemma}
\begin{proof}
It follows from Lemma \ref{lem3.2} that $\widehat{u}\in\mathbb{Z}^{\alpha,2}_{\widehat{p}\lambda,0,\text{loc}}\big( \mathfrak{R}^d\big)$  in view of  $\widehat{u}\in\mathbb{Z}_{p,q,w,0}^{\alpha,2}\left( \mathfrak{R}^d\right)$. According to \eqref{e3.4} and \eqref{e3.21}, if $R\geq 2^{-\frac{\alpha}{2}}\zeta_{0}$, by H\"{o}lder's inequality, we obtain
\begin{equation}\label{e3.24}
\begin{aligned}
&(|D^{2}u-(D^{2}u)_{(t^{\ast}-(\rho R)^{\frac{2}{\alpha}},t^{\ast})\times B^{+}_{\rho R}(x^{\ast})}|)_{(t^{\ast}-(\rho R)^{\frac{2}{\alpha}},t^{\ast})\times B^{+}_{\rho R}(x^{\ast})}\\
\lesssim&\rho^{-\frac{d}{\widehat{q}}}(|D^{2}u|^{\widehat{p}})_{(t^{\ast}-(\rho R)^{\frac{2}{\alpha}},t^{\ast})\times B^{+}_{\rho R}(x^{\ast})}^{\frac{1}{\widehat{p}}}.
\end{aligned}
\end{equation}

For $R< 2^{-\frac{\alpha}{2}}\zeta_{0}$, $$-\partial_t^\alpha \widehat{u}+\widetilde{a}^{ij}D_{ij}\widehat{u}=\widetilde{f},$$
where $\widetilde{f}:=\widehat{f}+(\widetilde{a}^{ij}-\widehat{a}^{ij})D_{ij}\widehat{u}$ and $\widetilde{a}^{ij}$ is the zero extension of $\overline{a}^{ij}$ in the whole space with
\begin{equation*}
\overline{a}^{ij}=\fint_{{Q_{2R}(t^{\ast},x^{\ast})\cap\mathfrak{R}^d_{+}}}a^{ij}d tdz.
\end{equation*}
By the H\"{o}lder inequality and \eqref{e3.5}, we deduce
\begin{equation}\label{e3.25}
\begin{aligned}
\left(|\widetilde{f}|^{\widehat{p}}\right)_{\left(\overline{t},t^{\ast}\right)\times B^{+}_{R}(x^{\ast})}^{\frac{1}{\widehat{p}}}
\leq&\left(|\widehat{f}|^{\widehat{p}}\right)_{\left(\overline{t},t^{\ast}\right)\times B^{+}_{R}(x^{\ast})}^{\frac{1}{\widehat{p}}}\\
&+\left(|\widetilde{a}^{ij}-\widehat{a}^{ij}|^{\gamma \widehat{p}}\right)_{(\overline{t},t^{\ast})\times B^{+}_{R}(x^{\ast})}^{\frac{1}{\gamma\widehat{p}}}
 \left(|D^{2}\widehat{u}|^{\lambda \widehat{p}}\right)_{(\overline{t},t^{\ast})\times B^{+}_{R}(x^{\ast})}^{\frac{1}{\lambda \widehat{p}}}.
\end{aligned}
\end{equation}
Combining with \eqref{e3.7} and \eqref{e3.25}, we deduce
\begin{equation}\label{e3.26}
\begin{aligned}
&\sum_{k=0}^{\infty}C_{k}\left(|\overline{f}|^{\widehat{p}}\right)_{(\overline{t},t^{\ast})\times B^{+}_{R}(x^{\ast})}^{\frac{1}{\widehat{p}}}\\
\lesssim&\sum_{k=0}^{\infty}C_{k}\left(|\widehat{f}|^{\widehat{p}}\right)_{(\overline{t},t^{\ast})\times B^{+}_{R}(x^{\ast})}^{\frac{1}{\widehat{p}}}+ \vartheta^{\frac{1}{\gamma\widehat{p}}}\sum_{k=0}^{\infty}C_{k}2^{\frac{k+2}{\gamma \widehat{p}}}\left(|D^{2}u|^{\lambda \widehat{p}}\right)_{(\overline{t},t^{\ast})\times B^{+}_{R}(x^{\ast})}^{\frac{1}{\lambda \widehat{p}}}.
\end{aligned}
\end{equation}
Therefore, by virtue of Lemma \ref{lem3.7}, \eqref{e3.23} follows from \eqref{e3.24} and \eqref{e3.26}. The proof is completed.
\end{proof}

\begin{lemma}\label{lem3.9}
Under Assumption \ref{assump2.3}, for any $u\in \widetilde{\mathbb{Z}} _{p, q, w, 0}^{\alpha, 2}\left(\mathfrak{R}^d_{+}\right )$ with compact support in $(0, T) \times B^{+}_{\zeta_{0}}$ satisfying \eqref{e3.22} in $\mathfrak{R}^d_{+}$,  there holds
\begin{equation}\label{e3.27}
\|\partial_t^\alpha u\|_{L_{p,q,w}(\mathfrak{R}^d_{+})}+\|D^2u\|_{L_{p,q,w}(\mathfrak{R}^d_{+})}\lesssim \|f\|_{L_{p,q,w}(\mathfrak{R}^d_{+})},
\end{equation}
where the hidden constant is dependent of $d, \tau, \alpha, p, q$ and $\varrho$.
\end{lemma}
\begin{proof}
Similar to Lemma \ref{lem3.3}, combining with  Lemma \ref{lem3.8}, we obtain
\begin{equation*}
\|D^2u\|_{p,q,w}\lesssim \|f\|_{p,q,w}.
\end{equation*}
Hence, by the fact
$-\partial_t^\alpha u=-a^{ij}(t,x)D_{ij}u+f$ in $\mathfrak{R}^d_{+}$,
we obtain \eqref{e3.27}. The proof is completed.
\end{proof}

\begin{lemma}\label{lem3.10}
Under Assumption \ref{assump2.3}, for any
$u\in \widetilde{\mathbb{Z}}_{p, q, w, 0}^{\alpha, 2}\left (\mathfrak{R}^d_{+}\right )$ satisfying
\begin{equation*}
-\partial_t^\alpha u+a^{ij}D_{ij}u+b^iD_{i}u+cu=f \qquad\text{in}\quad\mathfrak{R}^d_{+},
\end{equation*}
there holds
\begin{align}\label{e3.28}
\|\partial_t^\alpha u\|_{p,q,w}+\|D^2u\|_{p,q,w}\lesssim \|f\|_{p,q,w}+\|u\|_{p,q,w}.
\end{align}
Furthermore, we deduce
\begin{align}\label{e3.29}
 \|I^{\alpha}_tf\|_{ p,q,w}\lesssim T^{\alpha}\|f\|_{ p,q,w}
\end{align}
and 
\begin{align}\label{e3.30}
\|u\|_{p,q,w}\lesssim T^\alpha\|\partial_t^\alpha u\|_{p,q,w},
\end{align}
where the hidden constant is dependent of $d, \tau, \alpha, p, q, \varrho, \zeta_{0}$ and $\|\cdot\|_{p, q, w}=\|\cdot\|_{L_{p,q,w}(\mathfrak{R}^d_{+})}$.
\end{lemma}
\begin{proof}
Define
$$\widehat{b}^{1}(t,x)=b^{1}(t,|x^1|,x^{\prime})\mathrm{sgn}x^1,\quad \widehat{b}^{i}(t,x)=b^{i}(t,|x^1|,x^{\prime}),i\geq2,$$
$$\widehat{c}(t,x)=c(t,|x^1|,x^{\prime}).$$
Together with the extension of Lemma \ref{lem3.7} and \cite[Lemma 7.3]{dong0}, we have $\widehat{u}$ satisfies
$$-\partial_t^\alpha\widehat{u}+\widehat{a}^{ij}D_{ij}\widehat{u}=\widehat{f}-\widehat{b}^iD_i\widehat{u}-\widehat{c}\widehat{u}.$$


According to Lemma \ref{lem3.4} and Lemma \ref{lem3.9}, for $p=q$, we get
\begin{align*}
\|\partial_t^\alpha u\|_{p,w}+\|D^2u\|_{p,w}
&\leq \|\partial_t^\alpha \widehat{u}\|_{p,\widehat{w}}+\|D^2\widehat{u}\|_{p,\widehat{w}} \\
&\lesssim \|f\|_{p,w}+\|Du\|_{p,w}+\|u\|_{p,w}.
\end{align*}
It follows from the extrapolation theorem that \eqref{e3.28} holds for $p\neq q$.

Moreover, we deduce
$$\|u\|_{p,q,w}\lesssim\|I^\alpha_t\partial_t^\alpha \widehat{u}\|_{p,q,\widehat{w}}\lesssim T^\alpha\|\partial_t^\alpha u\|_{p,q,w}.$$ Therefore, \eqref{e3.30} holds. The proof is completed.
\end{proof}


\begin{proof}[\textbf{Proof of Theorem \ref{the2.1}}]
In order to prove \eqref{e2.4}, by \eqref{e3.28}, we need to verify
\begin{equation}\label{e3.31}
\|u\|_{p,q,w}\lesssim \|f\|_{p,q,w}.
\end{equation}
First, we extend  $\widehat{u},\widehat{f}$ to be zero for $t<0$. Consider  $a\leq0$, it means that $\widehat{u}$ satisfies the equation
$$-\partial_t^\alpha\widehat{u}+\widehat{a}^{ij}D_{ij}\widehat{u}+\widehat{b}^iD_i\widehat{f}+\widehat{c}\widehat{u}=\widehat{f}\quad\mathrm{in}\quad(a,T)\times\mathbb{R}^{d}.$$
Denote $\widehat{u\varphi_l}\in\mathbb{Z}_{p,q,\widehat{w},0}^{\alpha,2}((s_{n,l-1},s_{n,l+1})\times\mathbb{R}^{d})$  to be its extension with respect to $x^1$ direction. Similar to Lemma \ref{lem3.5}, we obtain $\widehat{u\varphi_l}$ satisfies
\begin{equation}\label{e3.32}
-\partial_{t}^{\alpha}(\widehat{u\varphi_l})+\widehat{a}^{ij}D_{ij}(\widehat{u\varphi_l})+\widehat{b}^{i}D_{i}(\widehat{u\varphi_l})+\widehat{c}(\widehat{u\varphi_l})=\widehat{f\varphi_l}+\widehat{g}_{l},
\end{equation}
where 
\begin{align*}
\widehat{g}_{l}(t,x)
&=\frac{\alpha}{\Gamma(1-\alpha)}\int_{-\infty}^t(t-s)^{-\alpha-1}\left(\varphi_l(t)-\varphi_l(s)\right)\widehat{u}(s,x)\chi_{s\leq s_{n,l}}ds.
\end{align*}
Hence, from \eqref{e3.17} and \eqref{e3.29}, we deduce
\begin{align*}
\|\widehat{g}_{l}\|_{p,q,\widehat{w}((s_{n,l-1},s_{n,l+1})\times\mathbb{R}^{d})}
&\leq \frac{2n\alpha}{T}\|I_{0}^{1-\alpha}\widehat{u}(t,x)\chi_{t\leq s_{n,l+1}}\|_{p,q,\widehat{w}((0,s_{n,l+1})\times\mathbb{R}^{d})}\\
&\leq\frac {2n\alpha}{T}\left(s_{n,l+1}\right)^{1-\alpha}\|\widehat{u}(t,x)\chi_{t\leq s_{n,l}}\|_{p,q,\widehat{w}((0,s_{n,l+1})\times\mathbb{R}^{d})}\\
&\leq T^{-\alpha}\|\widehat{u}\|_{p,q,\widehat{w}((0,s_{n,l})\times\mathbb{R}^{d})}.
\end{align*}
By \eqref{e3.30}, \eqref{e3.32} and the last inequality, 
it yields
\begin{align*}
&\|u\|_{p,q,w((s_{n,l},s_{n,l+1})\times\mathbb{R}^{d}_{+})}\\
\lesssim& \left(\frac{T}{n}\right)^{\alpha}\|\partial_{t}^{\alpha}(\widehat{u\varphi_l})\|_{p,q,\widehat{w}((s_{n,l-1},s_{n,l+1})\times\mathbb{R}^{d})} \\
\lesssim &\|\widehat{f\varphi_l}\|_{p,q,\widehat{w}((s_{n,l-1},s_{n,l+1})\times\mathbb{R}^{d})}+n^{-\alpha}\|\widehat{u}\|_{p,q,\widehat{w}((0,s_{n,l})\times\mathbb{R}^{d})} +\left(\frac{T}{n}\right)^{\alpha}\|\widehat{u}\|_{p,q,\widehat{w}((s_{n,l-1},s_{n,l+1})\times\mathbb{R}^{d})}\\
\lesssim&\|f\|_{p,q,w((0,s_{n,l+1})\times\mathbb{R}^{d}_{+})}+\|u\|_{p,q,w((0,s_{n,l})\times\mathbb{R}^{d}_{+}))},
\end{align*}
with  $n$ large enough. Therefore, by induction we derive \eqref{e2.4} from $\|u\|_{p,q,w((0,s_0)\times\mathbb{R}^{d}_{+})}=0$.

According to the result of Lemma \ref{lem3.5}, we get the existence of solutions for Eq. \eqref{e2.3}. The proof is completed.
\end{proof}

\section{Divergence form in the whole space }\label{sec4}

In this section, we consider divergence form equations  in the whole space and give the proof of Theorem \ref{the2.2}.

First, we introduce the mollification of $v(t,x)$ in the spatial variable as 
$$v^\epsilon(t,x)=\int_{\mathbb{R}^d}\varphi^\epsilon(x-y)v(t,y)\:dy,$$
where a smooth non-negative function $\varphi\in C_0^\infty(\mathbb{R}^d)$ satisfying supp$(\varphi)\subset B_1$ and $\displaystyle\int_{\mathbb{R}^d}\varphi=1$ with $\varphi^\epsilon(x)=\epsilon^{-d}\varphi(\frac{x}{\epsilon})$.

Furthermore, we have $v\in\mathcal{Z}_{p,q,w,0}^{\alpha,1}((a,T)\times\mathbb{R}^d)$ satisfies \eqref{e2.5}
and
\begin{equation*}
\begin{aligned}
&\int_{a}^{T}\int_{\mathbb{R}^d}I_{a}^{1-\alpha}v\partial_{t}\psi\:dxdt+\int_{a}^{T}\int_{\mathbb{R}^d}(-a^{ij}D_{j}vD_{i}\psi-a^{i}vD_{i}\psi+b^{i}D_{i}v\psi+cv\psi)\:dxdt\\
&=\int_{a}^{T}\int_{\mathbb{R}^d}(g\psi-f_{i}D_{i}\psi)\:dxdt,
\end{aligned}
\end{equation*}
where $f_{i},g\in L_{p,q,w}((a,T)\times\mathbb{R}^d)$ and $\psi\in C_0^\infty([a,T)\times\mathbb{R}^d)$.

\begin{lemma} \label{lem4.1}
Let $a^{ij}$ be constant, $R\in (0, \infty )$, $\rho \in (0, \frac{1}{4})$, $p, q\in$ $( 1, \infty )$, $\varrho\in$ $[1, \infty )$, and $[w]_{p,q}\leq \varrho$. Under Assumption \ref{assump2.2}, for 
$f=(f_1,\ldots,f_d)$, $f_i\in L_{p,q,w}(\mathfrak{R}^d)$, $i=1,\ldots,d$, there exists a unique solution $u \in \mathcal{Z}_{\widehat{p}, 0,loc}^{\alpha, 1}\left (\mathfrak{R}^d\right )$ of
\begin{equation}\label{eq4.1}
-\partial_t^\alpha u+D_i(a^{ij}D_{j}u)=D_if_i
\end{equation}
in $\mathfrak{R}^d$. Moreover, for $\rho \in(0,\frac{1}{4})$, we have
\begin{equation}\label{eq4.2}
\begin{aligned}
&(|Du-(Du)_{(t^{\ast}-(\rho R)^{\frac{2}{\alpha}},t^{\ast})\times B_{\tau\rho R}(x^{\ast})}|)_{(t^{\ast}-(\rho R)^{\frac{2}{\alpha}},t^{\ast})\times B_{\tau\rho R}(x^{\ast})} \\
\lesssim&\rho^\beta\sum_{k=0}^{\infty}2^{-k\alpha}\left(|D u|^{\widehat{p}}\right)_{(\widetilde{t},t^{\ast}) \times B_{\frac{\tau^{-1}R}{2}}(x^{\ast})}^{\frac{1}{\widehat{p}}}\\
&+(\rho^\beta+\rho^{\frac{-(d+\frac{2}{\alpha})}{\widehat{p}}})\sum_{k=0}^\infty C_k\left(|f|^{\widehat{p}}\right)_{(\overline{t},t^{\ast})\times B_{\tau^{-1}R}(x^{\ast})}^{\frac{1}{\widehat{p}}},
\end{aligned}
\end{equation}
with $\left(t^{\ast}, x^{\ast}\right) \in (0, T] \times \mathbb{R}^{d}$ and the hidden constant is only dependent of $d, \tau, \alpha,\widehat{p}$.
\end{lemma}
\begin{proof}
 First, if $a^{ij}=\tau$, by Lemma \ref{lem3.5}, there exist $v_{i} \in \mathbb{Z}_{p, q, w, 0}^{\alpha,2}\left(\mathfrak{R}^d\right)$  satisfying
\begin{equation}\label{eq4.3}
\partial_{t}^{\alpha} v_{i}-\Delta v_{i}=f_{i}    
\end{equation}
and
$$\left\|D v_{i}\right\|_{p, q, w}+\left\|D^{2} v_{i}\right\|_{p, q, w} \lesssim \left\|f_{i}\right\|_{p, q, w}.$$ 
Then, we get $u:=\sum\limits_{i=1}^{d} D_{i} v_{i} \in \mathcal{Z}_{p, q, w, 0}^{\alpha, 1}\left(\mathfrak{R}^d\right)$
is a solution of \eqref{eq4.3} and
$$\|u\|_{p, q, w}+\|D u\|_{p, q, w} \leq \sum_{i=1}^{d}\left\|D v_{i}\right\|_{p, q, w}+\sum_{i=1}^{d}\left\|D^{2} v_{i}\right\|_{p, q, w} \lesssim\|f\|_{p, q, w}.$$
Furthermore, we obtain
$$\left\|\partial_{t}^{\alpha} u\right\|_{\mathbb{Z}_{p, q, w}^{-1}\left(\mathfrak{R}^d\right)}+\|u\|_{p, q, w}+\|D u\|_{p, q, w} \lesssim \|f\|_{p, q, w}.
$$

To obtain the uniqueness of the solution, we assume that there exist two solutions $u_1$ and $u_2$ in $\mathcal{Z}_{p, q, w, 0}^{\alpha, 1}\left(\mathfrak{R}^d\right)$. Let $v=u_1 -u_2 \in \mathcal{Z}_{p, q, w, 0}^{\alpha, 1}\left(\mathfrak{R}^d\right)$ and  satisfy
$$\partial_{t}^{\alpha} v-\Delta v=0.$$
Then, by the definition of $v^{\varepsilon}\in \mathbb{Z}_{p, q, w, 0}^{\alpha, 2}\left(\mathfrak{R}^d\right)$, we get
$$I^{1-\alpha} v^{\varepsilon}(t, x)=\frac{1}{\Gamma(1-\alpha)} \int_{0}^{t}(t-s)^{-\alpha} v^{\varepsilon}(s, x) d s=\int_{\mathbb{R}^{d}} I^{1-\alpha} v(t, y) \varphi^{\varepsilon}(x-y) d y.$$
Thus, we deduce
\begin{align*}
&\int_{0}^{T} \int_{\mathbb{R}^{d}} I^{1-\alpha} v^{\varepsilon}(t, x) \Psi_{t}(t, x)dxdt+\int_{0}^{T} \int_{\mathbb{R}^{d}} D_{i} v^{\varepsilon} D_{i} \Psi dxdt \\
=&\int_{0}^{T} \int_{\mathbb{R}^{d}} I^{1-\alpha} v(t, y)\left(\int_{\mathbb{R}^{d}} \Psi_{t } \varphi^{\varepsilon}(x-y) d x\right) d y d t \\
&+\int_{0}^{T} \int_{\mathbb{R}^{d}} D_{i} v(t, y) \int_{\mathbb{R}^{d}} D_{i} \Phi(t, x)\varphi^{\varepsilon}(x-y) d x d y d t\\
=&\int_{0}^{T} \int_{\mathbb{R}^{d}}^{T} I^{1-\alpha} v \overline{\Psi}(t, x) dxdt+\int_{0}^{T} \int_{\mathbb{R}^{d}} D_{i} v D_{i} \overline{\Psi}=0,   
\end{align*}
where $\Psi \in C_{0}^{\infty}\left([0, T) \times \mathbb{R}^{d}\right) $ and $\overline{\Psi}(t, x):=\int_{\mathbb{R}^{d}} \Psi(t, y) \varphi^{\varepsilon}(y-x) d y$. 
Hence, we obtain $$\partial_{t}^{\alpha} v^{\varepsilon}-\Delta v^{\varepsilon}=0.$$
Together with the last equation and  Lemma \ref{lem3.5}, we have $v^{\varepsilon}=0$. Thus, we obtain $v=0$ as $\varepsilon \to 0$ and it follows from Lemma \ref{lem3.1} that \eqref{eq4.2} holds.

To handle general $a^{ij}$, let $\mathbf{B}=(a^{ij})$ be symmetric and positive definite. The, we obtain 
$\mathbf{B}^{\frac{1}{2}}$ is symmetric and invertible.
Set
$\overline{v}(t,z)=v(t,\mathbf{B}^{\frac{1}{2}}z)$ and  $z^{\ast}: = \mathbf{B}^{-\frac{1}{2}}(x^{\ast})$. And $\overline{v}(t,x)$ satisfies
$$\partial_t^\alpha\overline{v}-\Delta \overline{v}=\overline{f}\quad\mathrm{in}\quad\mathfrak{R}^d.$$
Then, we deduce
\begin{align*}
&(|D\overline{v}-(D\overline{v})_{Q_{\rho R}(t^{\ast},z^{\ast})}|)_{Q_{\rho R}(t^{\ast},z^{\ast})} \\
\lesssim&\rho^\beta\sum_{k=0}^{\infty}2^{-k\alpha}\left(|D\overline{v}|^{\widehat{p}}\right)_{(\widetilde{t},t^{\ast})\times B_{\frac{R}{2}}(z^{\ast})}+(\rho^\beta+\rho^{\frac{-(d+\frac{2}{\alpha})}{\widehat{p}}})\sum_{k=0}^\infty C_k\left(|\overline{f}|^{\widehat{p}}\right)_{(\overline{t},t^{\ast})\times B_{R}(z^{\ast})}^{\frac{1}{\widehat{p}}}.
\end{align*}
There exists $\tau \in (0,1)$ such that
$$B_{\tau R}(x^{\ast})\subset \mathbf{B}^{\frac{1}{2}}(B_R(z^{\ast}))\subset B_{\tau^{-1}R}(x^{\ast}).$$
Based on the result above, we deduce \eqref{eq4.2}. The proof is completed. 
\end{proof}

\begin{lemma}\label{lem4.2}
Under Assumption \ref{assump2.2}, there are $ \widehat{p} \in (1, \infty )$ and $\lambda \in (1, \infty )$ satisfying
$\widehat{p}<\widehat{p}\lambda<\min\{p,q\}$. If $u\in \mathcal{Z}_{p, q, w, 0}^{\alpha, 1}\left (\mathfrak{R}^d\right )$ has compact support in $(0, T) \times B_{\zeta_{0}}$ and satisfies
\begin{equation}\label{e4.1}
-\partial_t^\alpha u+D_i(a^{ij}(t,x)D_{j}u)=D_if_i
\end{equation}
in $\mathfrak{R}^d$, there holds
\begin{equation}\label{e4.2}
\begin{aligned}
&(|Du-(Du)_{(t^{\ast}-(\rho R)^{\frac{2}{\alpha}},t^{\ast})\times B_{\tau\rho R}(x^{\ast})}|)_{(t^{\ast}-(\rho R)^{\frac{2}{\alpha}},t^{\ast})\times B_{\tau\rho R}(x^{\ast})} \\
\lesssim&\rho^{-\frac{d}{\widehat{q}}}(|D^{2}u|^{\widehat{p}})_{(t^{\ast}-(\rho R)^{\frac{2}{\alpha}},t^{\ast})\times B_{\tau\rho R}(x^{\ast})}^{\frac{1}{\widehat{p}}}+
\rho^\beta\sum_{k=0}^{\infty}2^{-k\alpha}\left(|Du|^{\widehat{p}}\right)_{(\widetilde{t},t^{\ast}) \times B_{\tau^{-1}\frac{R}{2}}(x^{\ast})}\\
&+(\rho^\beta+\rho^{\frac{-(d+\frac{2}{\alpha})}{\widehat{p}}})\sum_{k=0}^{\infty}C_{k}\left(|f|^{\widehat{p}}\right)_{(\overline{t},t^{\ast})\times B_{\tau^{-1}R}(x^{\ast})}^{\frac{1}{\widehat{p}}}\\
\quad&+(\rho^\beta+\rho^{\frac{-(d+\frac{2}{\alpha})}{\widehat{p}}})\vartheta^{\frac{1}{\gamma\widehat{p}}}\sum_{k=0}^{\infty}C_{k}2^{\frac{k+2}{\gamma\widehat{p}}}\left(|Du|^{\lambda \widehat{p}}\right)_{(\overline{t},t^{\ast})\times B_{\tau^{-1}R}(x^{\ast})}^{\frac{1}{\lambda \widehat{p}}},
\end{aligned}
\end{equation}
where $\gamma =\frac{\lambda}{\lambda-1}$, $\widehat{q}=\frac{\widehat{p}}{\widehat{p}-1}$, $( t^{\ast}, x^{\ast}) \in (0, T] \times \mathbb{R}^{d}$. For $t \leq 0$, both $u$ and $f$ are equal to $0$.

Furthermore, we obtain
\begin{equation}\label{e4.3}
\|Du\|_{L_{p,q,w}(\mathfrak{R}^d)}\lesssim \|f\|_{L_{p,q,w}(\mathfrak{R}^d)},
\end{equation}
where the hidden constant is dependent of $d, \tau, \alpha, p, q$ and $\varrho$.
\end{lemma}
\begin{proof}
 For $ \widehat{p} \in (1, \infty )$ and $\lambda \in (1, \infty )$ satisfying
$\widehat{p}<\widehat{p}\lambda<\min\{p,q\}$, if $u\in\mathcal{Z}_{p,q,w,0}^{\alpha,1}\left( \mathfrak{R}^d\right)$, we deduce  $u\in\mathcal{Z}^{\alpha,1}_{\widehat{p}\lambda,0,\text{loc}}\big( \mathfrak{R}^d\big)$. 
According to the proof of Lemma \ref{lem3.2} and Lemma \ref{lem4.1}, the estimate \eqref{e4.2} follows. Hence,  from the definition of the dyadic sharp function and Lemma \ref{lem3.3}, it follows from \eqref{e3.9} that \eqref{e4.3} holds. The proof is completed.
\end{proof}

\begin{lemma}\label{lem4.3}
Let $R\in (0, \infty )$, $p, q\in$ $( 1, \infty )$, $\varrho\in$ $[1, \infty )$, and $[w]_{p,q}\leq \varrho$.
There are constants $\sigma_0$ such that $\sigma\geq\sigma_0$. Under Assumption \ref{assump2.2}, for $f_i, g \in L_{p,q,w}(\mathfrak{R}^d)$,
$f=(f_1,\ldots,f_d)$, $i=1,\ldots,d$,
if $u\in \mathcal{Z}_{p, q, w, 0}^{\alpha, 1}\left (\mathfrak{R}^d\right )$ has compact support in $(0, T) \times B_{\frac{\zeta_{0}}{\sqrt2}}$ and satisfies
\begin{equation}\label{e4.4}
-\partial_t^\alpha u+D_i(a^{ij}(t,x)D_{j}u)- \sigma u=D_if_i + g
\end{equation}
in $\mathfrak{R}^d$, then 
\begin{equation}\label{e4.5}
\sigma\|u\|_{L_{p,q,w}(\mathfrak{R}^d)}+\sqrt{\sigma}\|Du\|_{L_{p,q,w}(\mathfrak{R}^d)}\lesssim \sqrt{\sigma}\|f\|_{L_{p,q,w}(\mathfrak{R}^d)}+\|g\|_{L_{p,q,w}(\mathfrak{R}^d)},
\end{equation}
where the hidden constant is dependent of $d, \tau, \alpha, p, q$ and $\varrho$.
\end{lemma}
\begin{proof}
According to S. Agmon' idea, for $t\in(0,T)$, $x \in \mathbb{R}^d$, $z\in\mathbb{R}$, we set
$\widetilde{u}(t,x,z)=u(t,x)\phi(z)\cos(\sqrt{\sigma}z)$,
and $\phi$ is a $C_0^\infty(\mathbb{R})$ function with $\phi\not\equiv0$. Combining with 
\eqref{e4.3} and \eqref{e4.4}, by the proof of \cite[Lemma 5.5]{kr1}, we deduce \eqref{e4.5}. The proof is completed.
\end{proof}

\begin{lemma}\label{lem4.4}
Let $R\in (0, \infty )$, $p, q\in$ $( 1, \infty )$, $\varrho\in$ $[1, \infty )$, and $[w]_{p,q}\leq \varrho$.
Under Assumption \ref{assump2.2}, for $f_i, g \in L_{p,q,w}(\mathfrak{R}^d)$,
$f=(f_1,\ldots,f_d)$, $i=1,\ldots,d$, there are constants $\sigma_0$ such that $\sigma\geq\sigma_0$ and $u\in \mathcal{Z}_{p, q, w, 0}^{\alpha, 1}\left (\mathfrak{R}^d\right )$ satisfies
\begin{equation}\label{e4.6}
-\partial_t^\alpha u+D_i(a^{ij}D_{j}u + a^iu)+ b^iD_iu +cu - \sigma u=D_if_i + g
\end{equation}
in $\mathfrak{R}^d$, then we get
\begin{equation}\label{e4.7}
\sigma\|u\|_{p,q,w}+\sqrt{\sigma}\|Du\|_{p,q,w}\lesssim \sqrt{\sigma}\|f\|_{p,q,w}+\|g\|_{p,q,w}.
\end{equation}
Furthermore, if $u\in \mathcal{Z}_{p, q, w, 0}^{\alpha, 1}\left (\mathfrak{R}^d\right )$ satisfies
$$-\partial_{t}^{\alpha} u=D_{i} f_{i}+g $$
in $\mathfrak{R}^d$, then for any $\epsilon<1$,  we have 
\begin{equation}\label{e4.8}
\left\|\partial_{t}^{\alpha} u^{\epsilon}\right\|_{p, q, w}\lesssim \frac{1}{\epsilon}\|g\|_{p, q, w}+\|f\|_{p, q, w},
\end{equation}
where the hidden constant is dependent of $d, \tau, \alpha, p, q$ and $\varrho$,  $\|\cdot\|_{p, q, w}=\|\cdot\|_{L_{p,q,w}(\mathfrak{R}^d)}$.
\end{lemma}
\begin{proof}
First, we consider $p=q$ and $a^i=b^i=c=0$. By use a partition of unity argument in the spatial, we set $u_l(t,x)=u(t,x)\phi_l(x)$, where $\phi_l(x) \geq 0$, $\phi_l(x) \in C_0^\infty (\mathbb{R}^{d})$, $supp(\phi_l(x)) \subset B_{\frac{\zeta_{0}}{\sqrt{2}}}(x_l)$ and
$$1\leq\sum_{l=1}^{\infty}\left|\phi_{l}(x)\right|^{p} \leq C_{0}, \quad \sum_{l=1}^{\infty}\left| D_{x} \phi_{l}(x)\right|^{p} \leq C_{1},$$
where $C_{0}$ is dependent of $p$ and $C_{1}$ is dependent of $d, \alpha, p, \zeta_{0}$.
Then, $u_l$ satisfies
\begin{equation*}
\begin{aligned}
-\partial_{t}^{\alpha} u_{l} +D_{i}\left(a^{i j} D_{j} u_{l}\right)-\sigma u_{l} 
& =-\left(\partial_{t}^{\alpha} u\right) \phi_{l}+\left(D_{i}\left(a^{i j} \phi_{l} D_{j} u\right)+D_{i}\left(a^{i j} u D_{j} \phi_{l}\right)\right)-\sigma u \phi_{l} \\
& =D_{i}\left(f_{i} \phi_{l}+a^{i j} u D_{j} \phi_{l}\right)+g \phi_{l}-a^{i j} D_{i} \phi_{l} D_{j} u-f_{i} D_{i} \phi_{l}.
\end{aligned}    
\end{equation*}
It follows from Lemma \ref{lem4.3} that
\begin{equation*}
\begin{aligned}
\sigma\left\|u_{l}\right\|_{p, w}+\sqrt{\sigma}\left\|D u_{l}\right\|_{p, w}
\lesssim &\sqrt{\sigma}\left(\left\|\phi_{l} f\right\|_{p, w}+\left\|u D_{j} \phi_{l}\right\|_{p, w}\right)\\
&+\left(\left\|g \phi_{l}\right\|_{p, w}+\left\|D_{i} \phi_{l} D_{j} u\right\|_{p, w}+\left\|D_{i} \phi_{l} f_{i}\right\|_{p, w}\right),
\end{aligned}    
\end{equation*}
which implies that
\begin{equation*}
\begin{aligned}
\sigma^{p}\|u\|_{p, w}^{p}+\sigma^{\frac{p}{2}}\|D u\|_{p, w}^{p} 
\leq & C_{2}\sigma^{\frac{p}{2}}\|f\|_{p, w}^{p}+C_{3}\sigma^{\frac{p}{2}}\|u\|_{p, w}^{p}
\\&+C_{2}\|g\|_{p, w}^{p}+C_{3}\|D u\|_{p, w}^{p}+C_{3}\|f\|_{p, w}^{p}.
\end{aligned}    
\end{equation*}
It allows us to choose a sufficiently large $\sigma_0$ such that for $\sigma \geq \sigma_0$,
$$\sigma^{p}-C_3\sigma^{\frac{p}{2}}>\frac{\sigma^{p}}{2}\quad\mathrm{and}\quad \sigma^{\frac{p}{2}}-C_3>\frac{\sigma^{\frac{p}{2}}}{2}.$$
Hence, we have 
\begin{equation}\label{e4.9}
\sigma\|u\|_{p,w}+\sqrt{\sigma}\|Du\|_{p,w}\lesssim \sqrt{\sigma}\|f\|_{p,w}+\|g\|_{p,w}.
\end{equation}

If $a^i,b^i,c \neq 0$,  we obtain
$$-\partial_t^\alpha u+D_i(a^{ij}D_ju)-\sigma u=D_i(f_i-a^iu)+g-b^iD_iu-cu.$$
Similarly, we also obtain \eqref{e4.9} by choosing large $\sigma_0$ depending on $\tau^{-1}$. According to Lemma \ref{lem3.4}, \eqref{e4.7} follows from the extrapolation theorem for the case of $p\neq q$. 

In view of the proof of \cite[Lemma A.3.]{dong2}, we get \eqref{e4.8}.
The proof is completed.
\end{proof}

\begin{lemma}\label{lem4.5}
Let $p \in(1, \infty)$,  $-\infty<a<t^{\ast}<T<\infty$, and $u \in   \mathcal{Z}_{p, q, w, 0}^{\alpha, 1}((a, T) \times \Omega)$. Choose $\psi \in C^{\infty}(\mathbb{R})$, $\psi(t)=0$ with $t \leq t^{\ast}$ and $\left|\psi^{\prime}(t)\right| \leq C$ for $ t \in \mathbb{R}$, then we have  $\psi u \in \mathcal{Z}_{p, q, w, 0}^{\alpha, 1}\left(\left(t^{\ast}, T\right) \times \Omega\right)$  and
$$\partial_{t}^{\alpha}(\psi u)(t,x)=\partial_{t} I_{t^{\ast}}^{1-\alpha}(\psi u)(t,x)=\psi (t)\partial_{t} I_{a}^{1-\alpha} u(t,x)-H(t,x),$$
where
\begin{equation*}
H(t, x)= \frac{\alpha}{\Gamma(1-\alpha)} \int_{a}^{t}(t-s)^{-\alpha-1}\left(\psi(s)-\psi(t)\right) u(s, x) d s.
\end{equation*}
Moreover, we obtain
\begin{equation}\label{e4.10}
\|H\|_{L_{p,q,w}\left(\left(t^{\ast}, T\right) \times \Omega\right)} \lesssim\|u\|_{L_{p,q,w}((a, T) \times \Omega)}.
\end{equation}
\end{lemma}
\begin{proof}
First, we choose $f,g\in L_{p,q,w}((a,T)\times\Omega)$ and $u \in   \mathcal{Z}_{p, q, w, 0}^{\alpha, 1}((a, T) \times \Omega)$ satisfying
$$\partial_t^\alpha u=\partial_tI_a^{1-\alpha}u=\operatorname{div}F+G.$$   
It follows from the definition of mollification and the Fubini theorem that
$$\partial_{t}^{\alpha}u^{\epsilon}=\partial_{t}I_{a}^{1-\alpha}u^{\epsilon}=\mathrm{div}\:F^{\epsilon}+G^{\epsilon}.$$
Thus, we get $u^{\epsilon}\in\mathbb{Z}_{p,q,w,0}^{\alpha,2}((t_{0},T)\times\Omega)$.
It follows from \cite[Lemma 3.6]{dong4} that 
\begin{equation*}
\partial_{t}^{\alpha}(\psi u^\epsilon)(t,x)=\psi (t)\partial_{t} I_{a}^{1-\alpha} u^\epsilon(t,x)-\frac{\alpha}{\Gamma(1-\alpha)} \int_{a}^{t}(t-s)^{-\alpha-1}\left(\psi(s)-\psi(t)\right) u^\epsilon(s, x) ds.
\end{equation*}
Then, by using the dominated convergence theorem and $\epsilon\to 0$, we get
$$\int_{a}^{t}(t-s)^{-\alpha-1}(\psi(s)-\psi(t))u^{\epsilon}(s,x)\:ds \rightarrow \int_{a}^{t}(t-s)^{-\alpha-1}(\psi(s)-\psi(t))v(s,x)\:ds$$
in $L_{p,q,w}((a,T)\times\Omega)$, and
$$\partial_{t}I_{a}^{1-\alpha}u^{\epsilon}\to\partial_{t}I_{a}^{1-\alpha}u  \quad \mathrm{and}  \quad  \partial_{t}I_{t^{\ast}}^{1-\alpha}u^{\epsilon}\to\partial_{t}I_{t^{\ast}}^{1-\alpha}u  \quad\mathrm{in}\quad\mathbb{Z}_{p,q,w}^{-1}((a,T)\times\Omega).$$ 
Furthermore, by $\left|\psi^{\prime}(t)\right| \leq C$ and the proof of Theorem \ref{the2.1}, we obtain
\begin{align*}    
\|H\|_{L_{p,q,w}\left(\left(t^{\ast}, T\right) \times \Omega\right)} 
\leq& C  \Gamma(1-\alpha) \|I_{a}^{1-\alpha}u(t, x)\|_{L_{p,q,w}\left(\left(t^{\ast}, T\right) \times \Omega\right)}\\
\lesssim&\|u\|_{L_{p,q,w}((a, T) \times \Omega)}.
\end{align*}
Hence, we have \eqref{e4.10}. The proof is completed.
\end{proof}


\begin{proof}[\textbf{Proof of Theorem \ref{the2.2}}]
According to \eqref{e2.5}, we deduce
\begin{equation*}
-\partial_{t}^{\alpha} u+D_{i}\left(a^{i j} D_{j} u+a^{i} u\right)+b^{i} D_{i} u+c u - \sigma u=D_{i} f_{i}+g - \sigma u,\quad \text {in} \quad\mathfrak{R}^{d}.
\end{equation*}
It follows from Lemma \ref{lem4.4} that
\begin{equation*}
\sigma\|u\|_{p,q,w}+\sqrt{\sigma}\|Du\|_{p,q,w}\lesssim \sqrt{\sigma}\|f\|_{p,q,w}+\|g\|_{p,q,w}+\sigma\|u\|_{p,q,w},
\end{equation*}
for all $\sigma\geq \sigma_0$. Choosing $\sigma=\max\{1,\sigma_0\}$, we conclude
\begin{equation}\label{e4.11}
\|u\|_{p,q,w}+\|Du\|_{p,q,w}\lesssim \|f\|_{p,q,w}+\|g\|_{p,q,w}+\|u\|_{p,q,w}.
\end{equation}
Hence, in order to prove \eqref{e2.6}, we need to verify
\begin{equation}\label{e4.12}
\|u\|_{p,q,w}\lesssim\|f\|_{p,q,w}+\|g\|_{p,q,w}.
\end{equation}
Similar to Lemma \ref{e3.5}, set $s_{n,l}=\frac{lT}{n}$ with $-1\leq l \leq n-1$.  Choose smooth functions $\psi_l(t)$ satisfying $\psi_l(t)=0$ if $t\leq s_{n,l-1}$ and $\psi_l(t)=1$ if $t\geq s_{n,l}$ with  $|\varphi_l^{\prime}|\leq\frac{2n}{T}$ for $ 0\leq l\leq n-1$. 
By Lemma \ref{lem4.5} and \eqref{e2.5}, we get
$\psi_l u \in \mathcal{Z}_{p, q, w, 0}^{\alpha, 1}\left(\left(s_{n,l-1}, s_{n,l+1}\right) \times \mathbb{R}^{d}\right)$ and
\begin{equation}\label{e4.13}
\partial_{t}^{\alpha}(u\psi_{l})=D_{i}(a^{ij}D_{j}(u\psi_{l})+a^{i}(u\psi_{l})-f_{i}\psi_{l})+b^{i}D_{i}(u\psi_{l})+c(u\psi_{l})-g\psi_{l}-H(t,x),
\end{equation}
where $$H(t,x)=\frac{\alpha}{\Gamma(1-\alpha)}\int_{-\infty}^t(t-s)^{-\alpha-1}\left(\varphi_l(t)-\varphi_l(s)\right)u(s,x)\chi_{s\leq s_{n,l}}ds.$$
Thus, for $\epsilon <1$, together with \eqref{e4.8}, \eqref{e4.10} and \eqref{e4.11},  we obtain
\begin{align}\label{e4.14}
 \left\|\partial_{t}^{\alpha}(u\psi_{l})^{\epsilon}\right\|_{p,q,w(s_{n,l-1},s_{n,l+1})} \nonumber
\lesssim& \frac 1\epsilon\big(n\|f\|_{p,q,w(0,s_{n,l+1})}+n\|g\|_{p,q,w(0,s_{n,l+1})}+n\|u\|_{p,q,w(0,s_{n,l})}\\
& +\|u\|_{p,q,w(s_{n,l-1},s_{n,l+1})}\big)+nT^{-\alpha}\|u\|_{p,q,w(0,s_{n,l})},
\end{align}
where $\|\cdot\|_{p, q, w(s_{n,1}, s_{n,2})}=\|\cdot\|_{L_{p,q,w}\left((s_{n,1}, s_{n,2})\times\mathbb{R}^d\right)}$.

Next, in view of the triangle inequality, it yields
\begin{equation*} 
	\begin{aligned}
	 \|u\|_{p, q, w\left(s_{n,l}, s_{n,l+1}\right)}  
		\leq&\left\|\left(u \psi_{l}\right)^{\epsilon}-u \psi_{l}\right\|_{p, q, w\left(s_{n,l-1}, s_{n,l+1}\right)}+\left\|\left(u \psi_{l}\right)^{\epsilon}\right\|_{p, q, w\left(s_{n,l-1}, s_{n,l+1}\right)}. 
	\end{aligned}
\end{equation*}
The estimate \eqref{e4.14} implies
\begin{equation}\label{e4.15}
\begin{aligned}
 \|u\|_{p, q, w\left(s_{n,l}, s_{n,l+1}\right)}  
\lesssim & \epsilon\left\|D\left(u \psi_{l}\right)\right\|_{p, q, w\left(s_{n,l-1}, s_{n,l+1}\right)}+(\frac Tn)^\alpha\|\partial_t^\alpha(u\psi_{l})^\epsilon\|_{p,q,w(s_{n,l-1},s_{n,l+1})}\\
\lesssim & \|f\|_{p,q,w(0,s_{n,l+1})}+\|g\|_{p,q,w(0,s_{n,l+1})},
\end{aligned}
\end{equation}
with choosing  a large enough $m$ and a small enough $\epsilon$ in the last inequality. Therefore, it follows from induction that \eqref{e2.6} holds.
 
Furthermore, by Lemma \ref{lem3.5} for $\alpha \in (0,1)$, there exists a unique solution $u \in \mathbb{Z}_{p,q,w,0}^{\alpha,2}(\mathfrak{R}^d)$ for \eqref{e3.13}. Combining with the proof of \cite[Theorem 2.9]{dong1} for $\alpha \in (1,2)$, we deduce the existence of solutions for Eq. \eqref{e2.5}. The proof is completed.
\end{proof}

\section{Divergence form in the half space }\label{sec5}
First, we define
$\widehat{w}(t,x)=w_{1}(t)\widehat{w_{2}}(x)$, for $\widehat{w_{2}}(x)=w_{2}(|x^{1}|,x^{\prime})$, that is, $\widehat{w_{2}}(x) \in A_{q}(\mathbb{R}^{d})$ is the even extension of $w_2\in A_{q}(\mathbb{R}^{d}_{+})$ in the $x^1$ direction. 

\begin{lemma}\label{lem5.1}
Set $\widehat u(t,x):=u(t,|x^1|,x^{\prime})\mathrm{sgn}x^1$ for $u\in \widetilde{\mathcal{Z}}_{p, q, w, 0}^{\alpha, 1}(\mathfrak{R}^d_{+})$. Then, we deduce $\widehat{u} \in\mathcal{Z}_{p,q,\widehat{w},0}^{\alpha,1}(\mathfrak{R}^d)$ and
\begin{equation}\label{e5.1}
\|u\|_{\mathcal{Z}_{p,q,w}^{\alpha,1}(\mathfrak{R}^d_{+})}\leq\|\widehat{u}\|_{\mathcal{Z}_{p,q,\widehat{w}}^{\alpha,1}(\mathfrak{R}^d)}\leq2\|u\|_{\mathcal{Z}_{p,q,w}^{\alpha,1}(\mathfrak{R}^d_{+})}.
\end{equation}
\end{lemma}

\begin{proof}
The proof is similar to Lemma \ref{lem3.6}.
\end{proof}

\begin{lemma} \label{lem5.2}
Let $a^{ij}$ be constant, $R\in (0, \infty )$, $\rho \in (0, \frac{1}{4})$, $p, q\in$ $( 1, \infty )$, $\varrho\in$ $[1, \infty )$, and $[w]_{p,q}\leq \varrho$. Under Assumption \ref{assump2.3}, for 
$f=(f_1,\ldots,f_d)$, $f_i\in L_{p,q,w}(\mathbb{R}^d)$, $i=1,\ldots,d$, if $u \in \mathcal{Z}_{\widehat{p}, 0,loc}^{\alpha, 1}\left (\mathfrak{R}^d_+\right )$  satisfies 
\begin{equation}\label{eq5.2}
-\partial_t^\alpha u+D_i(a^{ij}D_{j}u)=D_if_i
\end{equation}
in $\mathfrak{R}^d_+$. Then, there exists $\rho \in(0,\frac{1}{4})$ satisfying
\begin{equation}\label{eq5.3}
\begin{aligned}
&(|Du-(Du)_{(t^{\ast}-(\rho R)^{\frac{2}{\alpha}},t^{\ast})\times B^+_{\tau\rho R}(x^{\ast})}|)_{(t^{\ast}-(\rho R)^{\frac{2}{\alpha}},t^{\ast})\times B^+_{\tau\rho R}(x^{\ast})} \\
\lesssim&\rho^\beta\sum_{k=0}^{\infty}2^{-k\alpha}\left(|D u|^{\widehat{p}}\right)_{(\widetilde{t},t^{\ast}) \times B^+_{\frac{\tau^{-1}R}{2}}(x^{\ast})}^{\frac{1}{\widehat{p}}}\\
&+(\rho^\beta+\rho^{\frac{-(d+\frac{2}{\alpha})}{\widehat{p}}})\sum_{k=0}^\infty C_k\left(|f|^{\widehat{p}}\right)_{(\overline{t},t^{\ast})\times B^+_{\tau^{-1}R}(x^{\ast})}^{\frac{1}{\widehat{p}}},
\end{aligned}
\end{equation}
with $\left(t^{\ast}, x^{\ast}\right) \in (0, T] \times \mathbb{R}^{d}_+$ and the hidden constant is only dependent of $d, \tau, \alpha,\widehat{p}$.
\end{lemma}
\begin{proof}
Similarly to Lemma \ref{lem3.7}, for $f_i \in L_{p,q,w}(\mathfrak{R}^d_+)$, $f=(f_1,\ldots,f_d)$,
we have the extension 
$$\widehat{a}^{11}(t,x)=a^{11}(t,|x^1|,x^{\prime}),\quad \widehat{a}^{1j}(t,x)=a^{1j}(t,|x^1|,x^{\prime})\mathrm{sgn}x^1,j\geq2,$$
$$\widehat{a}^{i1}(t,x)=a^{i1}(t,|x^1|,x^{\prime})\mathrm{sgn}x^1,i\geq2,\quad \widehat{a}^{ij}(t,x)=a^{ij}(t,|x^1|,x^{\prime}),i,j\geq2,$$
$$\widehat{f_1}(t,x)=f_1(t,|x^1|,x^{\prime})\mathrm{sgn}x^1, \quad \widehat{f_i}(t,x)=f_i(t,|x^1|,x^{\prime}),i\geq2,$$ and $\widehat{u}(t,x)=u(t, |x^1|,x^{\prime})\mathrm{sgn}x^1$ satisfies
\begin{equation*}
-\partial_t^\alpha \widehat{u}+D_i(\widehat{a}^{ij}D_{j}\widehat{u})=D_i\widehat{f_i}.
\end{equation*}
It follows from  Lemma \ref{lem4.1} and  Lemma \ref{lem5.1} that \eqref{eq5.3} holds. The proof is completed.
\end{proof}

\begin{lemma}\label{lem5.2}
Let $R\in (0, \infty )$, $p, q\in$ $( 1, \infty )$, $\varrho\in$ $[1, \infty )$, and $[w]_{p,q}\leq \varrho$. Under Assumption \ref{assump2.3}, there are constants $\sigma_0$ such that $\sigma\geq\sigma_0$
and $u\in \mathcal{Z}_{p, q, w, 0}^{\alpha, 1}\left (\mathfrak{R}^d_+\right )$ satisfies
\begin{equation}\label{e5.2}
-\partial_t^\alpha u+D_i(a^{ij}D_{j}u + a^iu)+ b^iD_iu +cu - \sigma u=D_if_i + g
\end{equation}
in $\mathfrak{R}^d_+$, then 
\begin{equation}\label{e5.3}
\sigma\|u\|_{p,q,w}+\sqrt{\sigma}\|Du\|_{p,q,w}\lesssim \sqrt{\sigma}\|f\|_{p,q,w}+\|g\|_{p,q,w},
\end{equation}
where the hidden constant is dependent of $d, \tau, \alpha, p, q$ and $\varrho$, $\|\cdot\|_{p, q, w}=\|\cdot\|_{L_{p,q,w}(\mathfrak{R}^d_+)}$.
\end{lemma}
\begin{proof}
Define
$$\widehat{a}^{1}(t,x)=a^{1}(t,|x^1|,x^{\prime})\mathrm{sgn}x^1,\quad \widehat{a}^{i}(t,x)=a^{i}(t,|x^1|,x^{\prime}),i\geq2,$$
$$\widehat{b}^{1}(t,x)=b^{1}(t,|x^1|,x^{\prime})\mathrm{sgn}x^1,\quad \widehat{b}^{i}(t,x)=b^{i}(t,|x^1|,x^{\prime}),i\geq2,$$
$$\widehat{c}(t,x)=c(t,|x^1|,x^{\prime}),\quad \widehat{g}(t,x)=g(t,|x^1|,x^{\prime}),$$
let $\widehat{u} \in\mathcal{Z}_{p,q,\widehat{w},0}^{\alpha,1}(\mathfrak{R}^d)$ and $\widehat{f}, \widehat{g}\in L_{p,q,\widehat{w}}(\mathfrak{R}^d)$ satisfy
\begin{equation*}
-\partial_t^\alpha \widehat{u}+D_i(\widehat{a}^{ij}D_{j}\widehat{u} + \widehat{a}^i\widehat{u})+ \widehat{b}^iD_i\widehat{u} +\widehat{c}\widehat{u} - \sigma \widehat{u}=D_i\widehat{f_i} + \widehat{g}.
\end{equation*}
Similarly to Lemmas \ref{lem4.2}-\ref{lem4.4} and Lemma \ref{lem5.1}, we obtain 
\begin{equation}\label{e5.4}
\begin{aligned}
\sigma\|u\|_{L_{p,q,w}(\mathfrak{R}^d_+)}+\sqrt{\sigma}\|Du\|_{L_{p,q,w}(\mathfrak{R}^d_+)}\leq&
\sigma\|\widehat{u}\|_{L_{p,q,\widehat{w}}(\mathfrak{R}^d)}+\sqrt{\sigma}\|D\widehat{u}\|_{L_{p,q,\widehat{w}}(\mathfrak{R}^d)}\\
\lesssim &\sqrt{\sigma}\|f\|_{L_{p,q,w}(\mathfrak{R}^d_+)}+\|g\|_{L_{p,q,w}(\mathfrak{R}^d_+)}.
\end{aligned}
\end{equation}
Therefore, the estimate \eqref{e5.3} follows. The proof is completed.
\end{proof}

\begin{proof}[\textbf{Proof of Theorem \ref{the2.3}}]
According to \eqref{e5.3},  we obtain 
\begin{equation*}
\|\widehat{u}\|_{p, q, \widehat{w}}+\|D\widehat{u}\|_{p, q, \widehat{w}}\lesssim \|\widehat{f}\|_{p, q, \widehat{w}}+\|\widehat{g}\|_{p, q, \widehat{w}}+\|\widehat{u}\|_{p, q, \widehat{w}},
\end{equation*}
for choosing $\sigma=\max\{\sigma_0, 1\}$, where $\|\cdot\|_{p, q, \widehat{w}}=\|\cdot\|_{L_{p,q,w}(\mathfrak{R}^d)}$.
Therefore, by Lemma \ref{lem5.1}, we deduce 
\begin{equation}\label{e5.5}
\|u\|_{L_{p,q,w}(\mathfrak{R}^d_+)}+\|Du\|_{L_{p,q,w}(\mathfrak{R}^d_+)}\lesssim \|f\|_{L_{p,q,w}(\mathfrak{R}^d_+)}+\|g\|_{L_{p,q,w}(\mathfrak{R}^d_+)}+\|u\|_{L_{p,q,w}(\mathfrak{R}^d_+)}.
\end{equation}
It is sufficient to prove that
\begin{equation}\label{e5.6}
\|u\|_{L_{p,q,w}(\mathfrak{R}^d_+)}\lesssim\|f\|_{L_{p,q,w}(\mathfrak{R}^d_+)}+\|g\|_{L_{p,q,w}(\mathfrak{R}^d_+)}.
\end{equation}
Similarly to the proof of Theorem \ref{the2.2}, we get
$\psi_l u \in \widetilde{\mathcal{Z}}_{p, q, w, 0}^{\alpha, 1}\left(\left(s_{n,l-1}, s_{n,l+1}\right) \times \mathbb{R}^{d}_+\right)$ and $\widehat{u\psi_{l}}$ satisfies \eqref{e4.14}.
Therefore, by Lemma \ref{lem5.1}, \eqref{e4.14} and \eqref{e4.15}, we deduce
\begin{equation}\label{e5.7}
\begin{aligned}
&\|u\|_{p, q, w\left(\left(s_{n,l-1}, s_{n,l+1}\right) \times \mathbb{R}^{d}_+\right)}\\
\lesssim & \epsilon\left\|D\left(u \psi_{l}\right)\right\|_{p, q, w\left(\left(s_{n,l-1}, s_{n,l+1}\right) \times \mathbb{R}^{d}_+\right)}+\left\|\partial_{t}^{\alpha}(u\psi_{l})^{\epsilon}\right\|_{p,q,\widehat{w}\left(\left(s_{n,l-1}, s_{n,l+1}\right) \times \mathbb{R}^{d}_+\right)}\\
\lesssim &\|f\|_{p,q,w\left(\left(0,T\right)\times \mathbb{R}^{d}_+\right)}+\|g\|_{p,q,w\left(\left(0,T\right)\times \mathbb{R}^{d}_+\right)},
\end{aligned}
\end{equation}
with choosing  a large enough $m$ and a small enough $\epsilon$ in the last inequality. The induction means that \eqref{e5.6} holds. Therefore, we get
\begin{equation*}
\left\|\partial_{t}^{\alpha} u\right\|_{\mathbb{Z}_{p, q, w}^{-1}(\mathfrak{R}^d_+)}+\|u\|_{L_{p,q,w}(\mathfrak{R}^d_+)}+\|D u\|_{L_{p,q,w}(\mathfrak{R}^d_+)} \lesssim\|f\|_{L_{p,q,w}(\mathfrak{R}^d_+)}+\|g\|_{L_{p,q,w}(\mathfrak{R}^d_+)}.
\end{equation*}
Consequently, by Theorem \ref{the2.2}, we deduce the existence of solutions for equation \eqref{e2.5} in the half space. The proof is completed.    
\end{proof}

\section*{Conflict of interest} 
The authors declare that they have no conflict of interest.

 \section*{Acknowledgements}
 This work is supported by the National Natural Science Foundation of China (12101142). 
 The authors are sincerely grateful to the referees for their helpful suggestions and comments.

\end{document}